\documentclass[11pt,a4paper, reqno]{amsart}
\usepackage[a4paper, left = 2cm,  hmargin={25mm,25mm},vmargin={25mm,25mm}]{geometry}
\usepackage[utf8]{inputenc}
\usepackage{geometry}
\usepackage[english]{babel}
\usepackage{graphicx}
\usepackage{color}
\usepackage{xcolor}
\usepackage{pdfpages}
\usepackage{amsbsy}
\usepackage{amssymb}
\usepackage{amsmath}
\usepackage{bm}
\usepackage{mathtools}
\usepackage{matlab-prettifier}
\usepackage{amsthm}
\usepackage{tikz-cd}
\usepackage{amsfonts}
\usepackage{array}
\usepackage{hyperref}
\usepackage[OT2,T1]{fontenc}
\usepackage{csquotes}
\usepackage{mathabx}
\usepackage{bm}
\usepackage{url}

\newcommand\m[1]{\begin{pmatrix}#1\end{pmatrix}} 
\newcommand{\Mod}[1]{\ (\mathrm{mod}\ #1)}

\DeclareSymbolFont{cyrletters}{OT2}{wncyr}{m}{n}
\DeclareMathSymbol{\Sha}{\mathalpha}{cyrletters}{"58}
\newcolumntype{P}[1]{>{\centering\arraybackslash}p{#1}}
\everymath{\displaystyle}
\newtheorem{theorem}{Theorem}[section]

\theoremstyle{definition}
\newtheorem{lemma}{Lemma}[section]

\theoremstyle{definition}
\newtheorem{defn}{Definition}[section]

\theoremstyle{definition}
\newtheorem{proposition}{Proposition}[section]

\newtheorem{remark}{Remark}

\newtheorem*{notation}{Notation}

\definecolor{lightgray}{gray}{0.5}
\newtheorem{innercustomthm}{Theorem}
\newenvironment{mythm}[1]
  {\renewcommand\theinnercustomthm{#1}\innercustomthm}
  {\endinnercustomthm}

\begin{document}

\title{A Fourier--Jacobi Dirichlet series for cusp forms on orthogonal groups}
\vspace{-0.5cm}
\author{Rafail Psyroukis}
\address{Department of Mathematical Sciences\\ Durham University\\
South. Rd.\\ Durham, DH1 3LE, U.K..}
\email{rafail.psyroukis@durham.ac.uk}
\maketitle
\vspace{-1cm}
\begin{abstract}
    We investigate a Dirichlet series involving the Fourier--Jacobi coefficients of two cusp forms $F,G$ for orthogonal groups of signature $(2,n+2)$. In the case when $F$ is a Hecke eigenform and $G$ is a Maass lift of a Poincar\'e series, we establish a connection with the standard $L$-function attached to $F$. What is more, we find explicit choices of orthogonal groups, for which we obtain a clear-cut Euler product expression for this Dirichlet series. Through our considerations, we recover a classical result for Siegel modular forms, first introduced by Kohnen and Skoruppa, but also provide a range of new examples, which can be related to other kinds of modular forms, such as paramodular, Hermitian, and quaternionic.
\end{abstract}
\section{Introduction}
Kohnen and Skoruppa in \cite{kohnen_skoruppa} considered the Dirichlet series
\begin{equation*}
    \mathcal{D}_{F,G}(s) = \zeta(2s-2k+4)\sum_{N=1}^{\infty} \langle \phi_N, \psi_N \rangle N^{-s},
\end{equation*}
where $\phi_N, \psi_N$ are the $N$th Fourier-Jacobi coefficients of two genus two Siegel cusp forms $F,G$ of integral weight $k\geq0$ and $\langle \;, \;\rangle$ denotes the Petersson inner product on the space of Jacobi forms of index $N$ and weight $k$. They showed that this Dirichlet series has an analytic continuation to $\mathbb{C}$ and satisfies a functional equation. Moreover, in the case when $F$ is a Hecke eigenform and $G$ is in the Maass space, they deduced that
\begin{equation*}
    \mathcal{D}_{F,G}(s) = \langle \phi_1, \psi_1\rangle Z_{F}(s),
\end{equation*}
where $Z_{F}(s)$ is the spinor $L$-function attached to $F$. In the case when $\phi_1$ does not vanish identically, this gives a way of proving the meromorphic continuation of $Z_{F}(s)$ to $\mathbb{C}$. Recently, Manickam in \cite{manickam} showed that the first Fourier-Jacobi coefficient of a degree $2$ cuspidal Siegel eigenform does not vanish identically, hence this is actually not a restriction.\\\\
It was not too long after the work of Kohnen and Skoruppa that Gritsenko in \cite{gritsenko} considered a same type Dirichlet series, however this time involving the Fourier--Jacobi coefficients of two Hermitian cusp forms of degree $2$. Gritsenko managed to prove that $\mathcal{D}_{F,G}(s)$ has an analytic continuation to $\mathbb{C}$ and that if $F$ is Hecke eigenform and $G$ is in the corresponding Maass space, $\mathcal{D}_{F,G}(s)$ is related to a degree $6$ $L$-function attached to $F$ (\cite[(5.3)]{gritsenko}). This $L$-function was first introduced by Gritsenko in \cite{gritsenko2}, in his study of a Dirichlet series involving the Fourier coefficients of a Hermitian cuspidal eigenform of degree $2$. It is related to the exterior square representation of $\textup{GL}_4$ (see \cite[Section 3.2.2]{pollack_shah_2} for example).\\\\
Even though the results in \cite{kohnen_skoruppa} and \cite{gritsenko} are very similar in nature, the methods used by the authors are quite different. In particular, in \cite{kohnen_skoruppa}, the authors prove their main result using a fundamental identity proved by Andrianov in \cite{andrianov}. This gives an Euler product expression for a certain Dirichlet series involving the Fourier coefficients of a Siegel cuspidal eigenform of degree $2$. In \cite{gritsenko} on the other hand, Gritsenko manages to factorise the polynomials defining the $p$-factors of the degree $6$ $L$-function attached to a Hermitian cuspidal eigenform of degree $2$. This factorisation happens in some parabolic Hecke rings which contain elements acting on the Fourier--Jacobi coefficients.\\\\
In this paper, we consider a Dirichlet series of the same type, attached to two cusp forms $F$ and $G$ for orthogonal groups of real signature $(2,n+2), \ n \geq 1$. We take $F$ to be a Hecke eigenform for the corresponding Hecke algebra and $G$ to be a Maass lift of a Poincar\'e series. The Maass space in this setting has been defined by Gritsenko in \cite{gritsenko_jacobi}, Sugano in \cite{sugano_lifting} and Krieg in \cite{krieg_jacobi}. Our main goal is to show how the method of Kohnen and Skoruppa in \cite{kohnen_skoruppa} can be extended in this setting and investigate the connection of this Dirichlet series with the standard $L$-function attached to $F$. The question of its analytic continuation and functional equation is quite different in nature and is the subject matter of our paper \cite{orthogonal_analytic}.\\\\
The main reason we expect such a connection is the accidental isogenies of orthogonal groups of signature $(2,n+2)$ with classical groups for small $n$. For example, for $n=1$, there is an isogeny with symplectic groups of degree $2$, for $n=2$, an isogeny with unitary groups of degree $2$, for $n=4$ with quaternionic groups of degree $2$ and for $n=6$ with the linear groups of Cayley numbers of degree $2$ (see for example \cite[page 5]{eisenstein_thesis}). Therefore, already the results of Kohnen and Skoruppa in \cite{kohnen_skoruppa} and Gritsenko in \cite{gritsenko} indicate a connection in the $n=1,2$ case. We note here that the isogeny between $\textup{SO}(2,3)$ and $\textup{Sp}_2$ in particular was recently exploited by Pollack and Shah in \cite{pollack_shah} in order to extend the integral of Kohnen and Skoruppa to apply for arbitrary cuspidal automorphic representations of $\textup{PGSp}_2$. \\

It should be noted that Gritsenko in \cite{gritsenko_dirichlet_orthogonal} initiated the study of a Dirichlet series involving the Fourier coefficients of an orthogonal modular form and its connection with the standard $L$-function attached to it, following factorisation methods similar to \cite{gritsenko}. Sugano in \cite{sugano} extended Gritsenko's result by proving an Euler product relation for a Dirichlet series involving twists of the Fourier coefficients by modular forms on a definite orthogonal group of lower rank. His work can be seen as an extension of Andrianov's work in \cite{andrianov} in the orthogonal setting. As we mentioned, our methods in this paper can be considered as an extension of the methods employed by Kohnen and Skoruppa in \cite{kohnen_skoruppa}, and therefore the result in \cite{sugano} by Sugano is pivotal. Finally, results proved by Shimura in \cite{shimura2004arithmetic} on the sets of solutions $\phi(x,x)=q$, where $\phi$ is a bilinear form and $q \in \mathbb{Q}^{\times}$, turn out to be crucial.\\

We remark that modified Dirichlet series involving Fourier-Jacobi coefficients can also be related to Euler products. This happens in \cite{heim} for example, where the case of $\textup{GSp}_2\times \textup{GL}_2$ is treated. Jointly with Bouganis in \cite{bouganis_psyroukis}, we partially extended this idea to $\textup{GU}(2,2)\times \textup{GL}_2$. The motivation is that these Dirichlet series admit integral representations of Rankin-Selberg type and therefore provide a way of proving the analytic properties of $L$-functions but also results on the algebraicity of their special values (see for example \cite{dummigan}, where the result in \cite{heim} finds a profound application). The isogenies we mentioned above mean that we can consider these as integral representations for the twisted standard $L$-functions of $\textup{SO}(2,3)$ and $\textup{SO}(2,4)$ respectively.\\



Let us now give a brief outline of our results. The setup is as follows: We start with an even positive definite symmetric 
matrix $S$ of rank $n \geq 1$ so that the lattice $L:=\mathbb{Z}^{n}$ is maximal. We then set
\begin{equation*}
    S_0 := \m{&&1\\ &-S&\\ 1&&}, \textup{ } S_1 := \m{&&1\\ &S_0&\\ 1&&}
\end{equation*}
of real signature $(1,n+1)$ and $(2,n+2)$ respectively. If now $K$ is a field containing $\mathbb{Q}$, we define the corresponding special orthogonal groups of $K$-rational points via
\begin{equation*}
    G^{*}_{K} = \{g \in \hbox{SL}_{n+2}(K) \mid g^{t}S_0g = S_0\}, \textup{ }G_{K} = \{g \in \hbox{SL}_{n+4}(K) \mid g^{t}S_1g = S_1\}.
\end{equation*}
The connected component of the identity $G_{\mathbb{R}}^{0}$ acts on a suitable upper half plane, which we will call $\mathcal{H}_{S} \subset \mathbb{C}^{n+2}$, and a suitable factor of automorphy $j(\gamma, Z)$ can be defined. Let us now denote by $\Gamma_{S}$ the intersection of $G_{\mathbb{R}}^{0}$ with the stabiliser of the lattice $L_1 = \mathbb{Z}^{n+4}$ and let $F : \mathcal{H}_S \longrightarrow \mathbb{C}$ be an orthogonal cusp form of integral even weight $k>n/2+2$ with respect to $\Gamma_{S}$ (see Definitions \ref{modular_form_defn} and \ref{fourier coefficients}). If now $Z = (\tau', z, \tau) \in \mathcal{H}_{S}$, with $\tau, \tau' \in \mathbb{C}$ and $z\in \mathbb{C}^{n+2}$, we can write
\begin{equation*}
    F(Z) = \sum_{m \geq 1}\phi_m(\tau, z)e^{2\pi i m \tau'},
\end{equation*}
and we call the functions $\phi_m(\tau, z)$ the Fourier--Jacobi coefficients of $F$. Now, if $G$ is another orthogonal cusp form with respect to some subgroup $\Gamma$, with $\widetilde{\Gamma}_S \leq \Gamma \leq \Gamma_{S}$ ($\widetilde{\Gamma}_S$ here denotes the discriminant kernel), and with corresponding Fourier--Jacobi coefficients $\{\psi_{m}\}_{m=1}^{\infty}$, the object of interest is
\begin{equation*}
    \mathcal{D}_{F,G}(s) = \sum_{m\geq 1}\langle \phi_{m},\psi_{m}\rangle m^{-s},
\end{equation*}
where $\langle \;,\; \rangle$ is a suitable inner product defined on the space of Fourier--Jacobi forms of certain weight and (lattice) index (see Definition \ref{inner product_fourier-jacobi}). This series converges absolutely for $\textup{Re}(s) > k+1$ and represents a holomorphic function in this domain (see Lemma \ref{convergence_dirichlet}). Our aim is to find an Euler product expression for this Dirichlet series.\\

We will now present our main Theorem \ref{dirichlet} in a simplified form. Let $F$ be an eigenform for some appropriately defined Hecke algebras (see Theorem \ref{dirichlet}) and with Fourier--Jacobi expansion as above. Let $\sigma$ denote the bilinear form induced by $S$, $(D,r) \in \widetilde{\textup{supp}}(L, \sigma)$ (see Definition \ref{support}) and define
\begin{equation*}
    \mathcal{P}_{k,D,r}(\tau', z, \tau) := \sum_{N \geq 1}(V_{N}P_{k,D,r})(\tau,z)e(N\tau'),
\end{equation*}
where $P_{k,D,r}$ is a specific Fourier--Jacobi form which reproduces the Fourier coefficients of any Fourier--Jacobi form, under the Petersson inner product (see Section \ref{Poincar\'e}). Also, $V_{N}$ denotes the index raising operator of Section \ref{v_N}, acting on Fourier--Jacobi forms. Let $L_0 = \mathbb{Z}^{n+2}$ and $L_0^{*} = S_0^{-1}L_0$. Assume now it exists $\xi \in V_0:=\mathbb{Q}^{n+2}$ such that $\{\xi^{t}S_0x \mid x \in L_0\}=\mathbb{Z}$ and $S_0[\xi] = -2D/q$, where $q$ denotes the level of $S$ (see Definition \ref{level}) and further assume $A(\xi) \neq 0$ for the Fourier coefficient of $F$. Consider then the algebraic subgroup of $G_{\mathbb{Q}}^{*}$, defined by
\begin{equation*}
    H(\xi)_{\mathbb{Q}} = \left\{g \in G_{\mathbb{Q}}^{*} \mid g\xi = \xi\right\}.
\end{equation*}
We consider the case when the class number $h:= \#(H(\xi)_{\mathbb{Q}}\backslash H(\xi)_{\mathbb{A}}/(H(\xi)_{\mathbb{A}}\cap C)$ equals $1$, where $C = \{x \in G_{\mathbb{A}}^{*} \mid xL_0=L_0\}$. Here, the subscript $_{\mathbb{A}}$ denotes the adelization of an algebraic linear group (see Section \ref{relation to the class number}). Then, our main Theorem \ref{dirichlet} reads:
\begin{mythm}{8.2}
With the notation and assumptions as above, let $\mathcal{P}$ denote a sufficiently large (specified) finite set of prime numbers. Define then the Dirichlet series
\begin{equation*}
    \mathcal{D}_{F,\mathcal{P}_{k,D,r},\mathcal{P}}(s) := \sum_{\substack{N=1\\(N,p)=1\forall p \in \mathcal{P}}}^{\infty} \langle \phi_N, V_NP_{k,D,r}\rangle N^{-s}.
\end{equation*}
We then have that $L_{\mathcal{P}}\left(\mathbf{1};s-k+(n+3)/2\right)\mathcal{D}_{F,P_{k,D,r},\mathcal{P}}(s)$ can be written as:
\begin{equation*}
    A(\xi)L_{\mathcal{P}}\left(F;s-k+(n+2)/2\right)\zeta_{\xi,\mathcal{P}}(s-k+n+1)\times \begin{cases}1 & \textup{if }n \textup{ odd} \\ \zeta_{\mathcal{P}}(2s-2k+n+2)^{-1} & \textup{if } n \textup{ even}
    \end{cases}.
\end{equation*}
Here, $L(-,s)$ denotes the standard $L$-function attached to orthogonal modular forms and $\mathbf{1}$ denotes the constant function on $H(\xi)_{\mathbb{A}}$. Also $\zeta_{\xi}(s)$ is a Dirichlet series counting the number of solutions of congruences (see Proposition \ref{first form of result}). For any zeta function (or Dirichlet series) appearing, we use the subscript $\mathcal{P}$ to denote that we don't take into account the $p$-factors (or $p$-terms) for $p \in \mathcal{P}$.
\end{mythm}
Note that in Theorem \ref{dirichlet}, we cover the case $h > 1$ as well. The rest of the paper (Section \ref{explicit examples}) is devoted in obtaining explicit examples where $h=1$. In those cases, we obtain a clear-cut Euler product expression for the specific Dirichlet series $\mathcal{D}_{F,\mathcal{P}_{k,-q,r}, \mathcal{P}}(s)$ (here we make a particular choice for $D$, i.e. $D=-q$ and take $\xi=(1,\textbf{0}, 1)^{t}$).\\\\
We find all instances when this happens in the rank $1$ case and we give the result in Theorem \ref{rank 1}. Note that in Theorem \ref{rank 1}, we recover the case of Kohnen and Skoruppa in \cite{kohnen_skoruppa} for our choice of $D$ (we take $S=2$), up to a finite number of Euler factors. What is more, we already produce new results related to $L$-functions attached to paramodular forms, as it is not clear how can one use the methods in \cite{kohnen_skoruppa} to produce such results for paramodular forms directly. Finally, in Theorem \ref{even rank theorem}, we provide at least one explicit example for each of the cases $n=2,4,6,8$. These examples can be related to modular forms for other classical groups, such as Hermitian ($n=2$) and quaternionic ($n=4$).\\\\
The structure of the paper is as follows: In Sections \ref{preliminaries} and \ref{fourier-jacobi forms} we give the definitions of the main objects of consideration, i.e. orthogonal groups, modular forms and Fourier--Jacobi forms. In Section \ref{v_N}, we calculate the action of the adjoint of the level-raising operator $V_N$ with respect to the Petersson inner product of Fourier--Jacobi forms. In Section \ref{Poincar\'e}, we define Poincar\'{e} series and in Section \ref{new section}, we define the Dirichlet series of interest and prove some preliminary results. Section \ref{relation to the class number} is devoted in showing how results of Shimura in \cite{shimura2004arithmetic} and \cite{shimura_diophantine_2006} can be used in our setting and in Section \ref{relation to sugano}, we use Sugano's theorem in \cite{sugano} to deduce our main Theorem \ref{dirichlet}. Finally, in Section \ref{explicit examples}, we obtain clear-cut Euler product expressions for some specific orthogonal groups. 
\begin{notation}
    We denote the space of $m\times n$ matrices with coefficients in a ring $R$ with $\textup{M}_{m,n}(R)$. If $n=m$, we often use the notation $\textup{M}_{n}(R)$. By $1_n$, we denote the $n\times n$ identity matrix. For any matrix $M \in \textup{M}_{n}(R)$, we denote by $\det(M), \textup{ tr}(M)$ the determinant and trace of $M$ respectively. By $\textup{GL}_n(R)$, we denote the matrices in $M_{n}(R)$ with non-zero determinant and by $\textup{SL}_n(R)$ the matrices with determinant $1$. For any vector $v$, we denote by $v^{t}$ its transpose. We also use the bracket notation $A[B] := \overline{B}^{t}AB$ for suitably sized complex matrices $A,B$. By $\textup{diag}(A_1, A_2, \cdots, A_n)$, we will denote the block diagonal matrix with the matrices $A_1, A_2,\cdots, A_n$ in the diagonal blocks. For a complex number $z$, we denote by $e(z) := e^{2\pi i z}$. Finally, let $\zeta(s)$ denote the Riemann zeta function and $\Gamma(s):= \int_{0}^{\infty}t^{s-1}e^{-t}\hbox{d}t$ denote the Gamma function.
\end{notation}
\section{Preliminaries}\label{preliminaries}
Let $V$ denote a finite dimensional vector space over $\mathbb{Q}$. We start with the following definition.
\begin{defn}
    A $\mathbb{Z}$-lattice $\Lambda$ is a free, finitely generated $\mathbb{Z}$-module, which spans $V$ over $\mathbb{Q}$.
\end{defn}
In the following, let $V := \mathbb{Q}^{n}$ and $L := \mathbb{Z}^{n}$ with $n \geq 1$. Assume $S$ is an even integral positive definite symmetric matrix of rank $n$. Here, even means $S[x] \in 2\mathbb{Z}$ for all $x \in L$. We define
\begin{equation*}
    S_0 := \m{&&1\\ &-S&\\1&&}, \ S_1 := \m{&&1\\&S_0&\\1&&}
\end{equation*}
of real signatures $(1,n+1)$ and $(2,n+2)$ respectively.
Let also $V_0:= \mathbb{Q}^{n+2}$ and $V_1 := \mathbb{Q}^{n+4}$ and consider the quadratic spaces $(V_0, \phi_0), \ (V_1, \phi_1)$, where
\begin{align*}
    \phi_i : V_i \times V_i &\longmapsto \mathbb{Q}\\
    (x,y) &\longmapsto \frac{1}{2}x^{t}S_iy,
\end{align*}
for $i= 0, 1$. 
We then have that $\phi:=\phi_{0} \mid_{V\times V}$ is just $(x, y) \longmapsto -x^{t}Sy/2$, and we make the assumption that $L = \mathbb{Z}^{n}$ is a maximal $\mathbb{Z}$-lattice with respect to $\phi$. Here, maximal means that the lattice $L$ is maximal among all $\mathbb{Z}$-lattices in $V$ on which $\phi[x]$ takes values in $\mathbb{Z}$, where we use the notation $\phi[x]:=\phi(x,x)$ for $x \in V$.\\

From \cite[Lemma 6.3]{shimura2004arithmetic}, we then obtain that $L_0:= \mathbb{Z}^{n+2}$ is a $\mathbb{Z}$-maximal lattice in $V_0$.\\

If now $K \supset \mathbb{Q}$ is a field, we define the corresponding special orthogonal groups of $K$-rational points via
\begin{equation*}
    G^{*}_{K} := \{g \in \textup{SL}_{n+2}(K) \mid g^{t}S_0g = S_0\},
\end{equation*}
\begin{equation*}
    G_{K} := \{g \in \textup{SL}_{n+4}(K) \mid g^{t}S_1g = S_1\}.
\end{equation*}
We view $G^{*}_{K}$ as a subgroup of $G_{K}$ via the embedding $g \longmapsto \m{1&&\\&g&\\ &&1}$.\\
Let now $\mathcal{H}_{S}$ denote one of the connected components of
\begin{equation*}
    \{Z \in V_0 \otimes_{\mathbb{Q}} \mathbb{C} \mid \phi_0[\textup{Im}Z] > 0\}.
\end{equation*}
In particular, if we denote by $\mathcal{P}_{S} := \{y'=(y_1,y,y_2) \in \mathbb{R}^{n+2} \mid y_1 >0, \phi_0[y']>0\}$, we choose
\begin{equation*}
    \mathcal{H}_S = \{z=x+iy \in V_0 \otimes_{\mathbb{R}} \mathbb{C} \mid y \in \mathcal{P}_S\}.
\end{equation*}
For a matrix $g \in \textup{Mat}(n+4, \mathbb{R})$, we write it as
\begin{equation*}
    g = \m{\alpha & a^t &\beta \\ b&A&c\\ \gamma&d^t&\delta},
\end{equation*}
with $A \in \textup{Mat}(n+2, \mathbb{R}), \alpha,\beta,\gamma,\delta \in \mathbb{R}$ and $a,b,c,d$ real column vectors. Now the map
\begin{equation}\label{action_orthogonal}
    Z \longmapsto g\langle Z\rangle = \frac{-\frac{1}{2}S_0[Z]b+AZ+c}{-\frac{1}{2}S_0[Z]\gamma+d^tZ+\delta}
\end{equation}
gives a well-defined transitive action of $G_{\mathbb{R}}^{0}$ on $\mathcal{H_{S}}$, where $G_{\mathbb{R}}^{0}$ denotes the identity component of $G_{\mathbb{R}}$. The denominator of the above expression is the factor of automorphy
\begin{equation*}
    j(g, Z) := -\frac{1}{2}S_0[Z]\gamma+d^tZ+\delta.
\end{equation*}
Let now $L_1 := \mathbb{Z}^{n+4}$ and define the groups
\begin{equation*}
    \Gamma(L_0):=\{g \in G_{\mathbb{Q}}^{*} \mid gL_0=L_0\}, 
\end{equation*}
\begin{equation*}
\Gamma(L_1) := \{g \in G_{\mathbb{Q}} \mid gL_1 = L_1\}.
\end{equation*}
Let also $\Gamma^{+}(L_0) = \Gamma(L_0) \cap G^{*,0}_{\mathbb{R}}$. Let now
\begin{equation*}
    \Gamma_{S} := G_{\mathbb{R}}^{0} \cap \Gamma(L_1),
\end{equation*}
and
\begin{equation*}
    \widetilde{\Gamma}_S:= \{M \in \Gamma_S \mid M \in 1_{n+4} + \textup{M}_{n+4}(\mathbb{Z})S_1\},
    \end{equation*}
the \textbf{discriminant kernel}.
\begin{defn}\label{modular_form_defn}
Let $k \in \mathbb{Z}$ and $\widetilde{\Gamma}_S\leq \Gamma \leq \Gamma_S$ a subgroup of finite index. A holomorphic function $F : \mathcal{H}_S \longrightarrow \mathbb{C}$ is called a modular form of weight $k$ with respect to $\Gamma$ if it satisfies the equation
\begin{equation*}
    (F|_k \gamma)(Z) := j(\gamma, Z)^{-k}F(\gamma\langle Z\rangle) = F(Z)
\end{equation*}
for all $\gamma \in \Gamma$ and $Z \in \mathcal{H}_S$. We will denote the set of such forms by $M_{k}(\Gamma)$.
\end{defn}
Let us now give a couple more definitions on lattices.
\begin{defn}
    Given a $\mathbb{Z}$-lattice $\Lambda$, equipped with a bilinear form $\sigma : V \times V \longrightarrow \mathbb{Q}$, where $V =  \Lambda \otimes_{\mathbb{Z}} \mathbb{Q}$, we define its dual lattice by
    \begin{equation*}
        \Lambda^{*} := \{x \in \Lambda \otimes_{\mathbb{Z}} \mathbb{Q} \mid \sigma(x, y) \in \mathbb{Z} \textup{ }\forall y \in \Lambda\}.
    \end{equation*}
\end{defn}
\begin{defn}\label{level}
    With the notation as above, the level of the lattice $\Lambda$ is given as the least positive integer $q$ such that $\frac{1}{2}q\sigma(x,x) \in \mathbb{Z}$ for every $x \in \Lambda^{*}$.
\end{defn}
Now, if $\widetilde{\Gamma}_S\leq \Gamma \leq \Gamma_S$, $F \in M_{k}(\Gamma)$ admits a Fourier expansion of the form (see \cite[(5.10)]{sugano_lifting})
\begin{equation}\label{orthogonal_fourier_expansion}
    F(Z) = \sum_{r \in L_0^{*}}A(r)e(r^{t}S_0Z),
\end{equation}
where $Z \in \mathcal{H}_S$. It is then Koecher's principle that gives us that $A(r)=0$ unless $r \in L_0^{*} \cap \overline{\mathcal{P_S}}$ ($\overline{\mathcal{P}_S}$ denotes the closure of $\mathcal{P_S}$, see \cite[Theorem 1.5.2]{eisenstein_thesis}). By \cite[Theorem 1.6.23]{eisenstein_thesis}, and because $L$ is maximal, we have the following definition for cusp forms.
\begin{defn}\label{fourier coefficients}
    If $\widetilde{\Gamma}_S\leq \Gamma \leq \Gamma_S$, $F \in M_{k}(\Gamma)$ is called a \textbf{cusp form} if it admits a Fourier expansion of the form
    \begin{equation*}
        F(Z) = \sum_{r \in L_0^{*} \cap \mathcal{P_S}}A(r)e(r^{t}S_0Z),
    \end{equation*}
    We denote the space of cusp forms by $S_k(\Gamma)$.
\end{defn}
Finally, we give the definition of the Maass space, due to Krieg in \cite[Section 5]{krieg_jacobi}.
\begin{defn}\label{maass space}
    Let $\widetilde{\Gamma}_S\leq \Gamma \leq \Gamma_S$. The \textbf{Maass space} $M_{k}^{*}(\Gamma)$ consists of all $F \in M_{k}(\Gamma)$, so that if their Fourier expansion is
    \begin{equation*}
        F(Z) = \sum_{r \in L_0^{*}}A(r)e(r^{t}S_0Z),
    \end{equation*}
    we have
    \begin{equation*}
        A(r) = \sum_{d \mid \gcd(\rho)} d^{k-1}A\left(\m{lm/d^2\\-S^{-1}\lambda/d\\1}\right),
    \end{equation*}
    where $r = S_0^{-1}\rho$, with $\rho = \m{m&\lambda&l}^{t}$. We also write $S_k^{*}(\Gamma)$ for the Maass space consisting of cusp forms. 
\end{defn}
\section{Fourier--Jacobi Forms}\label{fourier-jacobi forms}
In this section, we will define Fourier--Jacobi forms of lattice index. We mainly follow Mocanu's thesis in \cite{jacobi_lattice} and Krieg in \cite{krieg_jacobi}. \\\\
For now, assume that $V$ is a vector space of dimension $n<\infty$ over $\mathbb{Q}$ together with a positive definite symmetric bilinear form $\sigma$ and an even lattice $\Lambda$ in $V$, i.e. $\sigma(\lambda, \lambda) \in 2\mathbb{Z}$ for all $\lambda \in \Lambda$. We start with the following definitions:
\begin{defn}
    We define the \textbf{Heisenberg group} to be
    \begin{equation*}
        H^{(\Lambda,\sigma)}(\mathbb{R}) = \{(x,y,\zeta) \mid x,y \in \Lambda \otimes \mathbb{R}, \zeta \in S^{1}\},
    \end{equation*}
    where $S^{1}:=\{z \in \mathbb{C} \mid |z|=1\}$, equipped with the following composition law:
    \begin{equation*}
        (x_1,y_1,\zeta_1)(x_2,y_2,\zeta_2) := (x_1+x_2, y_1+y_2, \zeta_1\zeta_2e(\sigma(x_1,y_2))).
    \end{equation*}
    The \textbf{integral Heisenberg group} is defined to be $H^{(\Lambda,\sigma)}(\mathbb{Z}):= \{(x,y,1) \mid x,y \in \Lambda\}$ and in the following we drop the last coordinate for convenience.
\end{defn}
\begin{proposition}
    The group $\hbox{SL}_2(\mathbb{R})$ acts on $H^{(\Lambda, \sigma)}(\mathbb{R})$ from the right, via
    \begin{equation*}
        ((x,y,\zeta), A) \longmapsto (x,y,\zeta)^{A} := \left((x,y)A, \zeta e\left(\sigma[(x,y)A]-\frac{1}{2}\sigma(x,y)\right)\right).
    \end{equation*}
    where $(x,y)A$ denotes the formal multiplication of the vector $(x,y)$ with $A$, i.e. if $A =\m{a&b\\c&d}$, we have $(x,y)A := (ax+cy, bx+dy)$.
\end{proposition}
\begin{defn}
    The \textbf{real Jacobi} group associated with $(\Lambda, \sigma)$, denoted with $J^{(\Lambda, \sigma)}(\mathbb{R})$, is defined to be the semi-direct product of $\hbox{SL}_2(\mathbb{R})$ and $H^{(\Lambda, \sigma)}(\mathbb{R})$. The composition law is then
    \begin{equation*}
        (A,h)\cdot (A',h') := (AA', h^{A'}h').
    \end{equation*}
    We also define the \textbf{integral Jacobi group} to be the semi-direct product of $\hbox{SL}_{2}(\mathbb{Z})$ and $H^{(\Lambda, \sigma)}(\mathbb{Z})$ and we will denote it by $J^{(\Lambda, \sigma)}$.
\end{defn}
We are now going to define some slash operators, acting on holomorphic, complex-valued functions on $\mathbb{H} \times (\Lambda \otimes \mathbb{C})$.
\begin{defn}\label{k-actions}
    Let $k$ be a positive integer and $f : \mathbb{H} \times (\Lambda \otimes \mathbb{C}) \longrightarrow \mathbb{C}$ a holomorphic function. For $M =\m{a&b\\c&d} \in \hbox{SL}_{2}(\mathbb{R})$ , we define
    \begin{equation*}
        (f|_{k, (\Lambda, \sigma)} [M])(\tau, z) := (c\tau+d)^{-k}e^{-\pi i c\sigma(z,z)/(c\tau + d)}f\left(\frac{a\tau+b}{c\tau+d}, \frac{z}{c\tau+d}\right).
    \end{equation*}
    In the case when $M \in \hbox{GL}_2^{+}(\mathbb{R})$, we define the action by using $\det(M)^{-1/2}M$ instead of $M$. For $h = (x,y,\zeta) \in H^{(\Lambda, \sigma)}(\mathbb{R})$, we define
    \begin{equation*}
        (f \mid_{k, (\Lambda,\sigma)}h)(\tau,z): = \zeta \cdot e^{\pi i \tau \sigma(x, x)+2\pi i \sigma(x, z)}f(\tau, z+x\tau+y).
\end{equation*}
    Finally, for the action of $J^{(\Lambda, \sigma)}(\mathbb{R})$, we have
    \begin{equation*}
        (f, (A,h)) \longmapsto (f \mid_{k, (\Lambda, \sigma)}(A,h))(\tau, z) := (f \mid_{k, (\Lambda, \sigma)} A) \mid_{k, (\Lambda, \sigma)} h.
    \end{equation*}
\end{defn}
We now have the following definition (\cite[Definition 1.23]{jacobi_lattice}).
\begin{defn}\label{fourier-jacobi-defn}
Let $V_{\mathbb{C}} := V \otimes \mathbb{C}$ and extend $\sigma$ to $V_{\mathbb{C}}$ by $\mathbb{C}$-linearity. For $k$ a positive integer, a holomorphic function $f: \mathbb{H} \times V_{\mathbb{C}} \rightarrow \mathbb{C}$ (where $\mathbb{H}$ denotes the usual upper half plane) is called a Fourier--Jacobi form of weight $k$ with respect to $(\Lambda, \sigma)$ if the following hold:
\begin{itemize}
    \item For all $\gamma \in J^{(\Lambda, \sigma)}$ and $(\tau, z) \in \mathbb{H} \times V_{\mathbb{C}}$, we have
    \begin{equation*}
        (f \mid_{k, (\Lambda, \sigma)} \gamma) (\tau,z) = f(\tau,z).
    \end{equation*}
    \item $f$ has a Fourier expansion of the form
\begin{equation*}
    f(\tau, z) = \sum_{m \in \mathbb{Z}, r \in \Lambda^{*}, 2m \geq \sigma[r]} c_{f}(m,r)e(m\tau + \sigma(r, z)).
\end{equation*}
\end{itemize}
We denote the space of such forms by $J_{k}(\Lambda, \sigma)$. We say $f$ is a \textbf{Jacobi cusp form} if $c_f(m,r)=0$ when $2m= \sigma[r]$. We denote the space of Fourier--Jacobi cusp forms by $S_{k}(\Lambda, \sigma)$.
\end{defn}
\begin{remark}
    We could also characterise Fourier--Jacobi forms as modular forms transforming under the action of the parabolic subgroup of $\Gamma_{S}$, given by
    \begin{equation*}\label{parabolic subgroup orthogonal}
        \Gamma_{S,J} := \left\{\m{*&*\\0&D} \in \Gamma_{S} \mid D \in \textup{SL}_{2}(\mathbb{Z})\right\}.
    \end{equation*}
    This is true because there is an embedding $\iota : J_S \longrightarrow \Gamma_{S,J}$ (see \cite[p. 44]{eisenstein_thesis} or \cite[p. 27]{sugano_lifting}).
\end{remark}
We now have a notion of a scalar product for elements of $S_{k}(\Lambda, \sigma)$ (\cite[Definition 1.33]{jacobi_lattice}). 
\begin{defn}\label{inner product_fourier-jacobi}
    Let $\phi, \psi \in S_{k}(\Lambda, \sigma)$. If then $U \leq J^{(\Lambda,\sigma)}$ of finite index, we define their Petersson inner product via
    \begin{equation*}
        \langle \phi, \psi\rangle_{U} = \frac{1}{\left[J^{\Lambda}:U\right]}\int_{U \backslash \mathcal{H}\times (\Lambda \otimes \mathbb{C})} \phi(\tau,z)\overline{\psi(\tau, z)}v^{k}e^{-2\pi \sigma(y,y)v^{-1}}\textup{d}V,
    \end{equation*}
    where $\tau = u+iv$, $z = x+iy$ and $\textup{d}V := v^{-n-2}\textup{d}u\textup{d}v\textup{d}x\textup{d}y$. This inner product does not depend on the choice of $U$, so in what follows, we drop the subscript.
\end{defn}
We now specify to our case, by taking $\Lambda = L = \mathbb{Z}^{n}$ and $\sigma(x,y) = x^{t}Sy$ for all $x,y \in V$.\\

Let us discuss the Fourier--Jacobi expansion of orthogonal cusp forms of weight $k$ with respect to a subgroup $\widetilde{\Gamma}_{S} \leq \Gamma \leq \Gamma_S$. If we write $Z = (\tau', z, \tau) \in \mathcal{H}_{S}$ with $\tau',\tau \in \mathbb{C}, \ z \in \mathbb{C}^{n}$, we have that for $F \in S_k(\Gamma)$ and any $m \in \mathbb{Z}$:
\begin{equation*}
    F(\tau'+m, z, \tau) = F(\tau', z, \tau).
\end{equation*}
Hence, we can write
\begin{equation}\label{fourier-jacobi expansion_og}
    F(Z) = \sum_{m \geq 1}\phi_m(\tau, z)e^{2\pi i m \tau'},
\end{equation}
and we call the functions $\phi_m(\tau, z)$ the Fourier--Jacobi coefficients of $F$. We note that then $\phi_{m} \in S_{k}(\mathbb{Z}^{n}, m\sigma)$ (see \cite[(5.10), (5.11)]{sugano_lifting}).
\section{$V_{N}$ operator}\label{v_N}
Let now again $\sigma(x,y) = x^{t}Sy$ for $x,y \in V$.  Let $N \geq 1$ and define $M_{2}(\mathbb{Z})_N := \{g \in M_{2}(\mathbb{Z}) | \det g = N\}$. For any $M = \m{a&b\\c&d} \in M_{2}(\mathbb{Z})_N$ and $\tau \in \mathbb{H}$ (usual upper half plane), we define
\begin{equation*}
    M\langle \tau\rangle := \frac{a\tau+b}{c\tau+d}.
\end{equation*}
Given a Fourier--Jacobi cusp form $\phi \in S_k(\mathbb{Z}^{n}, \sigma)$, we define the operator
\begin{equation*}
    V_{N} : S_{k}(\mathbb{Z}^{n}, \sigma) \longrightarrow S_{k}(\mathbb{Z}^{n}, N\sigma),
\end{equation*}
given by
\begin{equation*}
(V_{N}\phi)(\tau, z) = N^{k-1}\sum_{M \in \textup{SL}_{2}(\mathbb{Z})\backslash M_{2}(\mathbb{Z})_{N}} (c\tau+d)^{-k}e^{-\pi i cNS[z]/(c\tau+d)}\phi\left(M\langle \tau \rangle, \frac{Nz}{c\tau+d}\right),
\end{equation*}
This is well-defined by \cite[Lemma 6.1]{sugano_lifting} or \cite[Definition 4.25]{jacobi_lattice}. Our aim in this Section is to compute its adjoint with respect to the scalar product of Fourier--Jacobi forms. This is the analogue of the main Proposition in \cite{kohnen_skoruppa}.\\

Now, if $\phi \in S_k(\mathbb{Z}^{n}, N\sigma)$, we will write its Fourier expansion in a form similar to Kohnen and Skoruppa in \cite[Section 2]{kohnen_skoruppa}. From Definition \ref{fourier-jacobi-defn}, we can write
\begin{equation*}
    \phi(\tau, z) = \sum_{\substack{m \in \mathbb{Z}, r \in \mathbb{Z}^{n}\\ 2Nm > r^{t}S^{-1}r}}c_{\phi}(m,r)e(m\tau + r^{t}z),
\end{equation*}
for some $c_{\phi}(m,r) \in \mathbb{C}$. Now, from the condition $2Nm > r^{t}S^{-1}r$ and the definition of the level (see Definition \ref{level}), we can write
\begin{equation*}
    Nmq - \frac{1}{2}qS^{-1}[r] = -D \implies m = \frac{\frac{1}{2}qS^{-1}[r]-D}{qN},
\end{equation*}
for some integer $D < 0$. Therefore, we can write
\begin{equation}\label{fourier-jacobi expansion}
    \phi(\tau, z) = \sum_{\substack{D \in \mathbb{Z}_{<0}, r \in \mathbb{Z}^{n}\\D \equiv \frac{1}{2}qS^{-1}[r] \mod qN}}c_{\phi}\left(D, r\right)e\left(\frac{\frac{1}{2}qS^{-1}[r]-D}{qN}\tau + r^{t} z\right).
\end{equation}
\begin{remark}\label{equality in coefficients}
    By adjusting \cite[Proposition 1.25]{jacobi_lattice} to the notation above, we have that 
    \begin{equation*}
        D=D'\textup{ and } S^{-1}r \equiv S^{-1}r' \pmod{N \mathbb{Z}^{n}} \implies c_{\phi}(D,r) = c_{\phi}(D',r') .
    \end{equation*}
\end{remark}
We are now ready to give the result concerning the adjoint of $V_N$. We use the notation that $s \equiv s' \Mod{NS\mathbb{Z}^{n}} \iff s-s' = NSu$ for some $u \in \mathbb{Z}^n$.
\begin{proposition}\label{adjoint}
The action of $V_{N}^{*}$, defined as the adjoint of $V_{N}$ with respect to the scalar product of Fourier--Jacobi forms, defined in \ref{inner product_fourier-jacobi}, is given by (for $\phi \in S_{k}(\mathbb{Z}^{n}, N\sigma)$)
\begin{equation*}
    \sum_{\substack{D \in \mathbb{Z}_{<0}, r \in \mathbb{Z}^{n}\\D \equiv \frac{1}{2}qS^{-1}[r] \mod qN}}c_{\phi}\left(D, r\right)e\left(\frac{\frac{1}{2}qS^{-1}[r]-D}{qN}\tau + r^{t} z\right) \longmapsto
\end{equation*}
\begin{multline*}
    \longmapsto \sum_{\substack{D<0, r \in \mathbb{Z}^{n}\\ D \equiv \frac{1}{2}qS^{-1}[r] \mod q}}\left(\sum_{d \mid N} d^{k-(n+1)} \sum_{\substack{s \mod dS\mathbb{Z}^{n}\\ D \equiv \frac{1}{2}qS^{-1}[s] \mod {qd}}}c_{\phi}\left(\left(\frac{N}{d}\right)^2D, \frac{N}{d}s\right)\right)\times\\\times e\left(\frac{\frac{1}{2}qS^{-1}[r]-D}{q}\tau+r^{t}z\right).
\end{multline*}
\end{proposition}
\begin{remark}
    The $(D,r)$ coefficient of $V_{N}^{*}\phi$ is independent of $r$.
\end{remark}
\begin{proof}
We start by writing
\begin{equation*}
    V_{N}\phi = N^{k/2-1} \sum_{A \in \textup{SL}_2(\mathbb{Z}) \backslash M_{2}(\mathbb{Z})_{N}}\phi_{\sqrt{N}} |_{k,(\mathbb{Z}^{n},N\sigma)} \left(\frac{1}{\sqrt{N}}A\right),
\end{equation*}
where we define $\phi_{\sqrt{N}}(\tau, z) := \phi\left(\tau, \sqrt{N}z\right)$. We remind here that if $A = \m{a&b\\c&d} \in \textup{GL}_{2}^{+}(\mathbb{R})$, we use $(\det A)^{-1/2}A$ in the $\mid_{k}-$action defined in Definition \ref{k-actions}. \\

We can then follow the proof in \cite[pages 554-556]{kohnen_skoruppa} line by line and arrive at the result that the adjoint $V_{N}^{*} : S_{k}(\mathbb{Z}^{n}, N\sigma) \longrightarrow S_{k}(\mathbb{Z}^{n}, \sigma)$ is given by
\begin{equation*}
    \phi \longmapsto N^{k/2-2n-1} \sum_{X \Mod{N\mathbb{Z}^{2n}}} \sum_{A \in \textup{SL}_2(\mathbb{Z}) \backslash M_{2}(\mathbb{Z})_N} \phi_{\sqrt{N}^{-1}} |_{k,(\mathbb{Z}^{n},\sigma)} \left(\frac{1}{\sqrt{N}}A\right) |_{k,(\mathbb{Z}^{n},\sigma)} X.
\end{equation*}
Let us now compute the action on the Fourier coefficients. We choose a set of representatives for $\textup{SL}_2(\mathbb{Z}) \backslash M_{2}(\mathbb{Z})_N$ of the form $\m{a&b\\0&d}$ with $ad = N$ and $0\leq b < d$. We then obtain that the above expression can be written as
\begin{equation*}
    N^{k/2-2n-1}\sum_{\lambda, \mu \in \mathbb{Z}^{n}/N\mathbb{Z}^{n}} \sum_{\substack{ad=N\\b \in \mathbb{Z}/d\mathbb{Z}}} \left(\frac{d}{\sqrt{N}}\right)^{-k} \phi\left(\frac{a\tau+b}{d}, \frac{z+\lambda\tau +\mu}{d}\right)e^{\pi i \tau S[\lambda]+2\pi i \lambda^{t}Sz}.
\end{equation*}
By then substituting the Fourier expansion
\begin{equation*}
    \phi(\tau, z) = \sum_{\substack{D \in \mathbb{Z}_{<0}, r \in \mathbb{Z}^{n}\\D \equiv \frac{1}{2}qS^{-1}[r] \mod qN}}c_{\phi}\left(D, r\right)e\left(\frac{\frac{1}{2}qS^{-1}[r]-D}{qN}\tau + r^{t}z\right),
\end{equation*}
for $\phi$, the above can be written as
\begin{multline*}
    N^{k/2-2n-1}\sum_{\lambda, \mu \in \mathbb{Z}^{n}/N\mathbb{Z}^{n}} \sum_{\substack{ad=N\\b \in \mathbb{Z}/d\mathbb{Z}}}\left(\frac{d}{\sqrt{N}}\right)^{-k}\sum_{D,r}c_{\phi}(D, r) e\left(\left(\frac{\frac{1}{2}qS^{-1}[r]-D}{qN}\cdot \frac{a}{d} + \frac{S[\lambda]}{2}+\frac{r^t \lambda}{d}\right)\tau +\right.\\\left.+\left(\frac{r^t z}{d}+\lambda^{t}Sz\right) +\frac{\frac{1}{2}qS^{-1}[r]-D}{qN}\cdot\frac{b}{d} + \frac{r^t \mu}{d} \right).
\end{multline*}
Now, the term
\begin{equation*}
    \sum_{\substack{b \in \mathbb{Z}/d\mathbb{Z}\\\mu \in \mathbb{Z}^{n}/N\mathbb{Z}^{n}}}e\left(\frac{\frac{1}{2}qS^{-1}[r]-D}{qN}\cdot\frac{b}{d} + \frac{r^t \mu}{d}\right)
\end{equation*}
is zero unless 
\begin{equation*}
    d \mid \frac{\frac{1}{2}qS^{-1}[r]-D}{qN} \textup{ and } d \mid r,
\end{equation*}
meaning it divides all of its components. In that case, the sum equals $dN^{n}$. The conditions then imply that we can replace $r \longmapsto dr$ and $D \longmapsto Dd^2$. The last one follows from the fact that $N = ad$ and so $d \mid N$ as well. Hence, we obtain that the expression equals
\begin{multline*}
    N^{k-n-1}\sum_{\lambda \in \mathbb{Z}^{n}/N\mathbb{Z}^{n}}\sum_{d \mid N} d^{1-k} \sum_{\substack{D<0, r\in \mathbb{Z}^{n}\\ D \equiv \frac{1}{2}qS^{-1}[r] \Mod {\frac{qN}{d}}}}c_{\phi}(d^2D, dr)\times \\ \times e\left(\left(\frac{\frac{1}{2}qS^{-1}[r]+\frac{1}{2}qS[\lambda]+qr^t \lambda-D}{q}\right)\tau +r^{t} z+\lambda^{t}Sz\right)=
\end{multline*}
\begin{multline}\label{almost final expression adjoint}
    = N^{k-n-1}\sum_{d\mid N}d^{1-k} \sum_{\lambda \in \mathbb{Z}^{n}/N\mathbb{Z}^{n}}\sum_{\substack{D<0, r\in \mathbb{Z}^{n}\\ D \equiv \frac{1}{2}qS^{-1}[r - S \lambda] \Mod {\frac{qN}{d}}}}c_{\phi}(d^2D, d(r-S \lambda))\times\\\times e\left(\frac{\frac{1}{2}qS^{-1}[r]-D}{q}\tau + r^{t} z\right),
\end{multline}
after setting $r \longmapsto r - S \lambda$. Now, as in \cite[p. 557]{kohnen_skoruppa}, we set 
\begin{equation}\label{lambda mod expression}
    \lambda \equiv t + \frac{N}{d}t' \Mod{N\mathbb{Z}^{n}},
\end{equation}
with $t \Mod {(N/d)\mathbb{Z}^n}$ and $t' \Mod {d\mathbb{Z}^{n}}$. We then have
\begin{equation*}
    S^{-1}(d(r-S \lambda)) \equiv S^{-1}(d(r - S t)) \mod N\mathbb{Z}^{n},
\end{equation*}
and
\begin{equation*}
    D \equiv \frac{1}{2}qS^{-1}[r - S t] \mod \frac{qN}{d}.
\end{equation*}
The first property can be seen easily and the second follows from the fact that we already have $\displaystyle{D \equiv \frac{1}{2}qS^{-1}[r - S \lambda] \mod \frac{qN}{d}}$ and also we can check
\begin{equation*}
    \frac{1}{2}qS^{-1}[r - S t] - \frac{1}{2}qS^{-1}[r - S \lambda] = \frac{1}{2}q[2r^t (\lambda-t) + S[t] - S[\lambda]].
\end{equation*}
So, it suffices to show that
\begin{equation*}
    \frac{1}{2}S[t] \equiv \frac{1}{2}S[\lambda] \mod {\frac{N}{d}}.
\end{equation*}
But this follows after we write $\lambda = t + N\alpha/d$ with $\alpha \in \mathbb{Z}^{n}$, from \eqref{lambda mod expression}.\\

Hence, because of Remark \ref{equality in coefficients}, expression \eqref{almost final expression adjoint} becomes (after replacing $d$ with $N/d$)
\begin{multline*}
    \sum_{d \mid N} d^{k-(n+1)} \sum_{t \in \mathbb{Z}^{n}/d\mathbb{Z}^{n}} \sum_{\substack{D<0, r \in \mathbb{Z}^{n}\\D \equiv \frac{1}{2}qS^{-1}[r-S t]\mod qd}}c_{\phi}\left(\left(\frac{N}{d}\right)^2D, \frac{N}{d}(r-S t)\right)\times \\\times e\left(\frac{\frac{1}{2}qS^{-1}[r]-D}{q}\tau + r^{t} z\right).
\end{multline*}
What is left to prove now (after fixing $D, r$ with the appropriate conditions) is
\begin{equation*}
    \sum_{d \mid N} d^{k-(n+1)} \sum_{\substack{t \in \mathbb{Z}^{n}/d\mathbb{Z}^{n}\\D \equiv \frac{1}{2}qS^{-1}[r-S t]\mod qd}}c_{\phi}\left(\left(\frac{N}{d}\right)^2D, \frac{N}{d}(r-S t)\right)=
\end{equation*}
\begin{equation*}
    = \sum_{d \mid N} d^{k-(n+1)} \sum_{\substack{t \mod dS\mathbb{Z}^{n}\\ D \equiv \frac{1}{2}qS^{-1}[t] \mod {qd}}}c_{\phi}\left(\left(\frac{N}{d}\right)^2D, \frac{N}{d}t\right).
\end{equation*}
This follows by setting $u = r - S t$ and then observing that $\displaystyle{D \equiv \frac{1}{2}qS^{-1}[u] \mod qd}$, 
\begin{equation*}
    r - S t \equiv r - S  t' \mod dS\mathbb{Z}^{n} \iff t \equiv t' \mod d\mathbb{Z}^{n}
\end{equation*}
and by using the fact that $c_{\phi}(D, s) = c_{\phi}(D, s')$ if $s \equiv s' \mod NS\mathbb{Z}^{n}$ (see Remark \ref{equality in coefficients}). So we can consider the entries $r-S t$ $\Mod {NS\mathbb{Z}^{n}}$ and all of these are different $\Mod {dS\mathbb{Z}^{n}}$.\qedhere\end{proof}
\section{Poincar\'e Series}\label{Poincar\'e}
In this Section, we define a special class of Fourier--Jacobi forms, called Poincar\'e series, which reproduce the Fourier coefficients of Jacobi forms under the Petersson scalar product. What is more, they generate the space of cusp forms and in turn the Maass space, as we will show in the next Section.
\begin{defn}\label{support}
    We define the support of the lattice $L=\mathbb{Z}^n$ with respect to the bilinear form $\sigma(x,y)=x^{t}Sy$ for $x,y \in V$ to be
    \begin{equation*}
        \textup{supp}(L, \sigma) := \left\{(D,r) \mid D \in \mathbb{Q}_{\leq 0}, r \in L^{*}, D \equiv \frac{1}{2}S[r] \pmod {\mathbb{Z}}\right\}.
    \end{equation*}
Now, if we write $r \longmapsto S^{-1}r$ with $r \in L$, we get
\begin{equation*}
    qD \equiv \frac{1}{2}qS^{-1}[r] \pmod {q\mathbb{Z}},
\end{equation*}
which then implies $D \in \frac{1}{q}\mathbb{Z}$. Hence, by writing $D \longmapsto D/q$, we get that equivalently 
\begin{equation*}
    \textup{supp}(L, \sigma) = \left\{(D/q,S^{-1}r) \mid D \in \mathbb{Z}_{\leq 0}, r \in L, D \equiv \frac{1}{2}qS^{-1}[r] \pmod {q}\right\}.
\end{equation*}
Let then
\begin{equation*}
    \widetilde{\textup{supp}}(L, \sigma) :=  \{(D,r) \in \mathbb{Z}_{\leq 0}\times L \mid (D/q, S^{-1}r) \in \textup{supp}(L, \sigma)\}.
\end{equation*}
and in the following this is the set we will use.
\end{defn}
\begin{defn}
    Let $(D,r) \in \widetilde{\textup{supp}}(L, \sigma)$. We then define the following complex valued function on $\mathbb{H} \times (L \otimes \mathbb{C})$:
    \begin{equation*}
        g_{D,r}(\tau,z) := e\left(\frac{\frac{1}{2}qS^{-1}[r]-D}{q}\tau + r^{t}z\right),
    \end{equation*}
    where $q$ is the level of $L$ (see Definition \ref{level}).
\end{defn}
\begin{defn}
    Let $(D,r) \in \widetilde{\textup{supp}}(L, \sigma)$ and set
    \begin{equation*}
        J^{(L, \sigma)}_{\infty}:=\left\{\left(\m{1&n\\0&1},(0,\mu)\right)\mid n\in \mathbb{Z}, \mu \in L\right\}.
    \end{equation*}
    The \textbf{Poincar\'e series} of weight $k$ for the lattice $(L, \sigma)$ is defined by
    \begin{equation*}
        P_{k,D,r}(\tau,z) := \sum_{\gamma \in J_{\infty}^{(L,\sigma)} \backslash J^{(L,\sigma)}} g_{D,r} \mid _{k, (L,\sigma)} \gamma (\tau,z).
    \end{equation*}
    If $k > n/2+2$, then $P_{k,D,r}$ is absolutely and uniformly convergent on compact subsets of $\mathbb{H} \times (L \otimes \mathbb{C})$ and defines an element of $S_k(L, \sigma)$ (see \cite[Theorem 2.3, (i)]{jacobi_lattice}). Moreover, from the same Theorem, it only depends on $S^{-1}r\pmod L $.
\end{defn}
We now have the following very important property, which can be again found in \cite[Theorem 2.3, (i)]{jacobi_lattice}.
\begin{proposition}\label{Poincar\'e_property}
    Let $(D,r) \in \widetilde{\textup{supp}}(L, \sigma)$. Then for any $f \in S_{k}(L,\sigma)$ with 
    \begin{equation*}
    f(\tau, z) = \sum_{\substack{D' \in \mathbb{Z}_{<0}, r' \in L\\\frac{1}{2}qS^{-1}[r'] \equiv D' \pmod{q}}}c_{f}\left(D', r'\right)e\left(\frac{\frac{1}{2}qS^{-1}[r']-D'}{q}\tau + (r')^{t} z\right),
    \end{equation*}
    we have
    \begin{equation*}
        \langle f, P_{k,D,r} \rangle = \lambda_{k,D}c_{f}(D, r),
    \end{equation*}
    for some constant $\lambda_{k,D} \in \mathbb{C}$ depending on $k$ and $D$.
\end{proposition}
\section{Dirichlet Series and relation with Fourier coefficients}\label{new section}
In this Section, we define the Dirichlet series of interest. Let $\widetilde{\Gamma}_S \leq \Gamma \leq \Gamma_S$ be a subgroup. Let $F, G \in S_k(\Gamma)$ with corresponding Fourier--Jacobi coefficients $\{\phi_m\}_{m=1}^{\infty}, \{\psi_m\}_{m=1}^{\infty}$ respectively (see \eqref{fourier-jacobi expansion_og}). We define the Dirichlet series
\begin{equation}\label{dirichlet_defn}
    \mathcal{D}_{F,G}(s) := \sum_{N=1}^{\infty}\langle \phi_N, \psi_N\rangle N^{-s},
\end{equation}
where $\langle \ , \ \rangle$ denotes the inner product of Definition \ref{inner product_fourier-jacobi}. We then have the following Lemma.
\begin{lemma}\label{convergence_dirichlet}
    $\mathcal{D}_{F,G}(s)$ converges absolutely for $\textup{Re}(s) > k+1$ and represents a holomorphic function in this domain.
\end{lemma}
\begin{proof}
    The proof is similar to \cite[Lemma 1]{kohnen_skoruppa}. We will show for $N \geq 1$ that
    \begin{equation*}
        \left \langle \phi_N, \psi_N \right \rangle = \mathcal{O}(N^{k}),
    \end{equation*}
    with the constant depending only on $F, G$. Indeed, fix $(\tau, z) \in \mathbb{H} \times \mathbb{C}^{n}$ and write $\tau = u+iv, \textup{ } z = x+iy$. If $F(q) = \sum_{N=1}^{\infty}\phi_{N}(\tau,z)q^{N}$, with $q=e^{2\pi i \tau'}$, we have by Cauchy's integral formula that
    \begin{equation*}
        \phi_{N}(\tau,z) = \oint_{|q|=r}\frac{F(q)}{q^{N+1}}\textup{d}q,
    \end{equation*}
    for any $0 < r < e^{-\pi S[y]/v}$. The bounds follow from the fact that $S_0[\textup{Im}Z]>0$ (here $Z = (\tau',z,\tau) \in \mathcal{H}_S$). If now $\tau' = u'+iv'$, the integral can be written as
    \begin{equation*}
        \phi_{N}(\tau,z) = \int_{0}^{1}F(Z)e^{-2\pi i N \tau'}\textup{d}u',
    \end{equation*}
    for any $v' > S[y]/2v$. But now $\left|F(Z)\right|\left(S_0[\textup{Im}Z]/2\right)^{k/2}$ is bounded on $\mathcal{H}_S$ from \cite[II, Lemma 3.28]{hauffe21}, say by a constant $C>0$. Therefore, after choosing $v'=S[y]/2v+1/N$, we have
    \begin{equation*}
        \left|\phi_{N}(\tau,z)\right| \leq Ce^{2\pi}\int_{0}^{1}\left(S_0[\textup{Im}Z]/2\right)^{-k/2}e^{\pi N S[y]/v}\textup{d}u' = Ce^{2\pi}\left(\frac{v}{N}\right)^{-k/2}e^{ \pi N S[y]v}.
    \end{equation*}
    Similarly for $\psi_N$ and then the claim follows from the definition of the inner product in Definition \ref{inner product_fourier-jacobi}.
\end{proof}
From now on, we let $k>n/2+2$ be an \textbf{even} integer. Let then $F \in S_k(\Gamma_S)$ and write $\phi_N$ for its Fourier--Jacobi coefficients. For $(D,r) \in \widetilde{\textup{supp}}(L, \sigma)$, let also
    \begin{equation}\label{poincare in maass}
    \mathcal{P}_{k,D,r}(\tau', z,\tau) := \sum_{N \geq 1}V_{N}P_{k,D,r}(\tau,z)e(N\tau').
    \end{equation}
    From \cite[Corollary 6.7]{sugano_lifting}, we have that $\mathcal{P}_{k,D,r} \in S_{k}^{*}(\widetilde{\Gamma}_S)$ (see Definition \ref{maass space}). Here, by abusing notation, we write $P_{k,D,r}$ to actually denote the Poincar\'e series $\displaystyle{\frac{1}{\lambda_{k,D}}P_{k,D,r}}$, with the quantities defined in Proposition \ref{Poincar\'e_property}. Then, we have the Dirichlet series
    \begin{equation*}
        \mathcal{D}_{F,\mathcal{P}_{k,D,r}}(s) := \sum_{N=1}^{\infty} \langle \phi_{N}, V_{N}P_{k,D,r}\rangle N^{-s} = \sum_{N=1}^{\infty} \langle V_{N}^{*}\phi_{N}, P_{k,D,r}\rangle N^{-s}.
    \end{equation*}
\begin{remark}
Even though $F$ and $\mathcal{P}_{k,D,r}$ are taken invariant with respect to different modular groups, we note here that the proof of Lemma \ref{convergence_dirichlet} is still valid, because $F \in S_k(\Gamma_{S}) \subset S_k(\widetilde{\Gamma}_S)$.
\end{remark}
\begin{remark}\label{maass space generators}
A corollary of \cite[Theorem 2.3 (i)]{jacobi_lattice} is that $S_k(L, \sigma)$ is generated by $\left\{P_{k,D,r} \mid (D,r) \in \widetilde{\textup{supp}}(L, \sigma), \ S^{-1}r \pmod L\right\}$ (cf. \cite[Corollary 2.4]{jacobi_lattice}). This observation, together with \cite[Corollary 6.7]{sugano_lifting} gives us that the Maass space $S_{k}^{*}(\widetilde{\Gamma}_S)$ is generated by $\left\{\mathcal{P}_{k,D,r} \mid (D,r) \in \widetilde{\textup{supp}}(L, \sigma), \ S^{-1}r \pmod L\right\}$. It is therefore enough to consider $G$ to be a Poincar\'{e} series in \eqref{dirichlet_defn}.
\end{remark}
We will now relate the $N$th term of  $\mathcal{D}_{F,\mathcal{P}_{k,D,r}}(s)$ with the Fourier coefficients of $F$.
\begin{proposition}\label{inner_product}
    With the notation as above and $N \geq 1$ we have:
    \begin{equation*}
  \langle V_N^* \phi_N, P_{k,D,r} \rangle   = \sum_{d \mid N} d^{k-(n+1)} \sum_{\substack{s \mod dS\mathbb{Z}^{n}\\ D \equiv \frac{1}{2}qS^{-1}[s] \mod {qd}}}A\left(\frac{N}{d}\left(\frac{\frac{1}{2}qS^{-1}[s]-D}{qd}, S^{-1}s, d\right)\right).
\end{equation*}
\end{proposition}
\begin{proof}
$F$ admits a Fourier expansion of the form (see equation \eqref{orthogonal_fourier_expansion})
\begin{equation*}
    F(Z) = \sum_{\tilde{r} \in L_0^{*}}A(\tilde{r})e({\widetilde{r}}^{t}S_0Z) = \sum_{N=1}^{\infty} \phi_{N}(\tau, z)e(N\tau'),
\end{equation*}
where $Z = (\tau', z, \tau) \in \mathcal{H}_S$ and $\tilde{r} = (m, r, N)$ with $r \in L^{*}$. We can then write 
\begin{equation*}
    \phi_{N}(\tau, z) = \sum_{m \in \mathbb{Z}, r \in L^{*}}A(m,r,N)e(m\tau - r^{t}Sz) = \sum_{m \in \mathbb{Z}, r \in \mathbb{Z}^{n}}A(m, S^{-1}r, N)e(m\tau-r^{t} z).
\end{equation*}
But now $\m{-1_2 &&\\&1_n&\\&&-1_2} \in \Gamma_{S}$ and therefore if $Z=(\tau',z,\tau) \in \mathcal{H}_S$, we have
\begin{equation*}
    F\left((\tau',-z,\tau)\right) = (-1)^kF((\tau',z,\tau)) = F((\tau',z,\tau)),
\end{equation*}
which then implies $A(m, r, N) = A(m, -r, N)$ for all $m, N \in \mathbb{Z}, \ r \in L^{*}$. Therefore, after setting $r \longmapsto -r$, we can re-write the above as
\begin{equation}\label{fourier-jacobi-fourier1}
    \phi_{N}(\tau, z) = \sum_{m \in \mathbb{Z}, r \in \mathbb{Z}^{n}}A(m, S^{-1}r, N)e(m\tau+r^{t} z).
\end{equation}
Now, a priori, we can write  
\begin{equation}\label{expansion fourier-jacobi2}
    \phi_N(\tau, z) = \sum_{D,r}c_{\phi_{N}}\left(D,r\right)e\left(\frac{\frac{1}{2}qS^{-1}[r]-D}{qN}\tau+r^tz\right),
\end{equation}
with $(D,r)$ as in equation \eqref{fourier-jacobi expansion}. 
Hence, from Proposition \ref{adjoint} and Proposition \ref{Poincar\'e_property}, we have
\begin{align*}
    \langle V_N^* \phi_N, P_{k,D,r} \rangle &= \sum_{d \mid N} d^{k-(n+1)} \sum_{\substack{s \Mod {dS\mathbb{Z}^{n}}\\ D \equiv \frac{1}{2}qS^{-1}[s] \Mod {qd}}}c_{\phi}\left(\left(\frac{N}{d}\right)^2D, \frac{N}{d}s\right) \\&=\sum_{d \mid N} d^{k-(n+1)} \sum_{\substack{s \Mod {dS\mathbb{Z}^{n}}\\ D \equiv \frac{1}{2}qS^{-1}[s] \Mod {qd}}}A\left(\frac{N}{d}\left(\frac{\frac{1}{2}qS^{-1}[s]-D}{qd}, S^{-1}s, d\right)\right),
\end{align*}
because after setting $r \longmapsto Ns/d, D \longmapsto N^2D/d^2$ in \eqref{expansion fourier-jacobi2}, we obtain from \eqref{fourier-jacobi-fourier1}
\begin{equation*}
    c_{\phi_{N}}\left(\frac{N^2}{d^2}D, \frac{N}{d}s\right) = A\left(\frac{N}{d}\left(\frac{\frac{1}{2}qS^{-1}[s]-D}{qd}, S^{-1}s, d\right)\right).\qedhere
\end{equation*}
\end{proof}
\section{Relation to the class number}\label{relation to the class number}
In this Section, our goal is to bring $\mathcal{D}_{F, \mathcal{P}_{k,D,r}}(s)$ into a form similar to the one in \cite[page 553]{kohnen_skoruppa}. We first need some definitions of the adelized groups and of genus and classes of lattices.\\\\
Let $V$ denote any finite dimensional vector space over $\mathbb{Q}$ of dimension $n \geq 1$. For each prime $p$ (including infinity), we define $V_p := V \otimes_{\mathbb{Q}} \mathbb{Q}_p$. For a $\mathbb{Z}$-lattice $L$ in $V$, we denote by $L_p:= L \otimes_{\mathbb{Z}}\mathbb{Z}_p$. This coincides with the $\mathbb{Z}_p$-linear span of $L$ in $V_p$.
\begin{proposition}\label{lattices}
Let $L$ be a fixed $\mathbb{Z}$-lattice in $V$. Then, the following are true:
\begin{itemize}
    \item If $M$ is another $\mathbb{Z}$-lattice, then $L_p = M_p$ for almost all $p$. Moreover, $L \subset M$ if $L_p \subset M_p$ for all $p$ and $L=M$ if $L_p = M_p$ for all $p$.
    \item For all $p<\infty$, let $N_p \in V_p$ denote a $\mathbb{Z}_p$-lattice such that $N_p = L_p$ for almost all $p<\infty$. Then, there is a $\mathbb{Z}$-lattice $M$ in $V$ such that $M_p = N_p$ for all $p < \infty$.
\end{itemize}
\begin{proof}
    See \cite[Lemma 9.2]{shimura2004arithmetic}.
\end{proof}
\end{proposition}
\begin{defn}
    We define the adelizations $V$ and $\textup{GL}(V)$, as follows:
    \begin{equation*}
        V_{\mathbb{A}} = \left\{v \in \prod_{p\leq \infty}V_p \mid v_p \in L_p \textup{ for almost all } p < \infty\right\},
    \end{equation*}
    \begin{equation*}
        \textup{GL}(V)_{\mathbb{A}} = \left\{\alpha \in \prod_{p\leq \infty}\textup{GL}(V_p) \mid \alpha_pL_p = L_p \textup{ for almost all } p < \infty\right\}.
    \end{equation*}
\end{defn}
Here $L$ is an arbitrary $\mathbb{Z}$-lattice, and the above definitions do not depend on the choice of it, by virtue of Proposition \ref{lattices}.
\begin{remark}
When we identify $V$ with $\mathbb{Q}^{n}$, we identify $\textup{GL}(V)$ with $\textup{GL}_{n}(\mathbb{Q})$. Then, we identify $\textup{GL}_{n}(\mathbb{Q})_{\mathbb{A}}$ with $\textup{GL}_{n}(\mathbb{Q}_{\mathbb{A}})$, where $\mathbb{Q}_{\mathbb{A}}$ is the usual ring of adeles.
\end{remark}
\begin{remark}
    Given $x \in \textup{GL}(V)_{\mathbb{A}}$, we have that $x_pL_{p} = L_p$ for almost all $p < \infty$. By Proposition \ref{lattices}, there is a $\mathbb{Z}$-lattice $M$ such that $M_p = x_pL_{p}$ for all $p$. We denote this lattice by $xL$. Hence, $xL$ is the lattice which has the property $(xL)_p = x_pL_p$ for all $p < \infty$.
\end{remark}
\begin{defn}
    Let $G \subset \hbox{GL}_{n}(\mathbb{Q})$ any algebraic linear group. For any field $K$ containing $\mathbb{Q}$, we denote by $G_{K}$ the group of $K$-rational points in $G$. For any prime number $p$, we abbreviate by $G_p$ the group $G_{\mathbb{Q}_p}$. Let also $G_{\infty}$ denote $G_{\mathbb{R}}$. We then define
    \begin{equation*}
        G_{\mathbb{A}} := \left\{x \in \prod_{p\leq \infty} G_p \mid x_pL_p = L_p \textup{ for almost all } p\right\}.
    \end{equation*}
    Moreover, we use the notation $G_{\mathbb{A},f}$ to denote the finite part of $G_{\mathbb{A}}$.
\end{defn}
\begin{defn}
    Let $L$ denote a $\mathbb{Z}$-lattice in $V$ and $G \subset \hbox{GL}_n(\mathbb{Q})$ any algebraic group. Then, for any $x \in G$, $xL$ is also a $\mathbb{Z}$-lattice in $V$. We define the \textbf{class} of $L$ to be $\{xL \mid x \in G\}$. Similarly, for any $x \in G_{\mathbb{A}}$, $xL$ is again a $\mathbb{Z}$-lattice in $V$. The set $\{xL \mid x \in G_{\mathbb{A}}\}$ is called the \textbf{genus} of $L$.
\end{defn}
Now, the genus of $L$ can be decomposed into disjoint union of classes. If $C = \{x \in G_{\mathbb{A}} \mid xL=L\}$, the map $xC \longmapsto xL$ gives a bijection between $G_{\mathbb{A}}/C$ and the genus of $L$, so gives a bijection between $G\backslash G_{\mathbb{A}}/C$ and the set of classes contained in the genus of $L$. In general, if $U$ is an open subgroup of $G_{\mathbb{A}}$, we call the number $\# (G\backslash G_{\mathbb{A}}/U)$ the \textbf{class number} of $G$ with respect to $U$. \\\\
Now, with all the main definitions out of the way, we want to deal with the vectors appearing in Proposition \ref{inner_product}. For a $\mathbb{Z}$-lattice $\Lambda$ in $V_0$, $\mathfrak{b}$ a fractional ideal of $\mathbb{Q}$ and $q \in \mathbb{Q}^{\times}$, we define
    \begin{equation*}
        \Lambda[q,\mathfrak{b}] := \left\{x \in V_0 \mid \phi_0[x] = q \textup{ and } \phi_0(x, \Lambda) = \mathfrak{b}\right\}.
    \end{equation*}
We then have the following Lemma.
\begin{lemma}\label{equivalence_classes}
    Let $D \in \mathbb{Z}_{\leq 0}$ and define the set
    \begin{equation*}
        \Xi_{d}:= \left\{ \xi = \left(\frac{\frac{1}{2}qS^{-1}[s]-D}{qd}, S^{-1}s, d\right)^{t} \mid  s \mod dS\mathbb{Z}^{n}, \ D \equiv \frac{1}{2}qS^{-1}[s] \mod {qd}\right\}.
    \end{equation*}
    We then have
    \begin{equation*}
        \Xi_{d} \subset L_0\left[-\frac{D}{q}, \frac{1}{2}\mathbb{Z}\right]
    \end{equation*}
    for all $d \geq 1$ coprime to $D$.
\end{lemma}
\begin{proof}
    Firstly, for any vector $\xi \in \Xi_{d}$, we can directly compute $\phi_0[\xi] = -D/q$. It now remains to show that the $\mathbb{Z}$-ideal
    \begin{equation*}
        \left\{\xi^{t}S_0x \mid x \in \mathbb{Z}^{n+2}\right\}
    \end{equation*}
    equals $\mathbb{Z}$ (because the bilinear form is $\phi_0 = S_0/2$). But for any basis vector $e_i$ of the lattice $\mathbb{Z}^{n+2}$, we have
    \begin{equation*}
        \xi^{t}S_0e_i = 
    \begin{cases}
         d & \textup{if }i=1\\
         -s_{i-1} & \textup{if } 2 \leq i \leq n+1\\
         \frac{\frac{1}{2}qS^{-1}[s]-D}{qd} & \textup{if } i=n+2
    \end{cases},
    \end{equation*}
    so the above ideal is contained in $\mathbb{Z}$ by the conditions that define the set $\Xi_{d}$. Now, if it were equal to $k\mathbb{Z}$ for some $k \geq 1$, then we would have
    \begin{equation*}
        k \mid d, \ k\mid s, \textup{ and } k \mid \frac{\frac{1}{2}qS^{-1}[s]-D}{qd}.
    \end{equation*}
    This would then imply $k \mid D$, and so $k=1$, as $d, \ D$ are assumed to be coprime.
\end{proof}
We now have the following very important Lemma.
\begin{lemma}\label{class_number_g}
    The class number of $G_{\mathbb{Q}}^{*}$, defined as $\# (G^{*}_{\mathbb{Q}}\backslash G^{*}_{\mathbb{A}}/C)$, where 
\begin{equation}\label{stabiliser}
    C = \{x \in G^{*}_{\mathbb{A}} \mid xL_0 = L_0\},
\end{equation}
is 1.
\end{lemma}
\begin{proof}
    This is shown in \cite[Remark 2.4, (5)]{shimura_diophantine_2006}, which is an improvement of \cite[Theorem 9.26]{shimura2004arithmetic}, as the technical assumptions are weakened.
\end{proof}
Assume now it exists $\xi \in L_0\left[-\frac{D}{q}, \frac{1}{2}\mathbb{Z}\right]$ and consider the algebraic subgroup of $G_{\mathbb{Q}}^{*}$
\begin{equation*}
    H(\xi)_{\mathbb{Q}} = \left\{g \in G_{\mathbb{Q}}^{*} \mid g\xi = \xi\right\}.
\end{equation*}
We note that if $W := \{x \in V_0 \mid \phi_0(x, \xi)=0\}$ and $\psi:= \phi_0 \mid _{W}$, then 
\begin{equation*}
    H(\xi)_{\mathbb{Q}} = \hbox{SO}^{\psi}(W)=\{\alpha \in \textup{GL}(W) \mid \psi[\alpha w] = \psi[w], \ \forall w \in W\}.
\end{equation*}
We now have the following Proposition, which is a special case of \cite[Theorem 2.2]{shimura_diophantine_2006}.
\begin{proposition}\label{genus-classes}
    There is a bijection 
    \begin{equation*}
        L_0\left[-\frac{D}{q}, \frac{1}{2}\mathbb{Z}\right]/\Gamma(L_0) \longleftrightarrow H(\xi)_{\mathbb{Q}}\backslash H(\xi)_{\mathbb{A}}/(H(\xi)_{\mathbb{A}}\cap C),
    \end{equation*}
    which is given as follows: \\\\
    If $k \in L_0\left[-\frac{D}{q}, \frac{1}{2}\mathbb{Z}\right]$, then there is some $\alpha \in G^{*}_{\mathbb{Q}}$ such that $k = \alpha \xi$ (Witt's theorem, \cite[Lemma 1.5 (ii)]{shimura2004arithmetic}). We then assign the $H(\xi)_{\mathbb{Q}}$-class of $\alpha^{-1}L_0$ to $k$. In particular, this then gives
    \begin{equation*}
    \# \left(L_0\left[-\frac{D}{q}, \frac{1}{2}\mathbb{Z}\right]/\Gamma(L_0)\right)  = \# (H(\xi)_{\mathbb{Q}}\backslash H(\xi)_{\mathbb{A}}/(H(\xi)_{\mathbb{A}}\cap C)) .
    \end{equation*}
\end{proposition}
\begin{proof}
    First of all, $\Gamma(L_0)$ acts on $L_0\left[-\frac{D}{q}, \frac{1}{2}\mathbb{Z}\right]$. This can be seen because if $\gamma \in \Gamma(L_0), x \in L_0\left[-\frac{D}{q}, \frac{1}{2}\mathbb{Z}\right]$, then $\phi_0[\gamma x] = \phi_0[x] = -D/q$ and $\phi_0(\gamma x, L_0) = \phi_0(\gamma x, \gamma L_0) = \phi_0(x, L_0) = \frac{1}{2}\mathbb{Z}$.\\
    
    The remaining assertions follow from \cite[Theorem 2.2, (iv)]{shimura_diophantine_2006}, because $G_\mathbb{Q}^{*}$ has class number one (Lemma \ref{class_number_g}) and $G_\mathbb{Q}^{*} \cap C = \Gamma(L_0)$.
\end{proof}
Let us now write
\begin{equation}\label{xi_i reps}
    L_0\left[-\frac{D}{q}, \frac{1}{2}\mathbb{Z}\right] = \bigsqcup_{i=1}^{h}\Gamma(L_0)\xi_i.
\end{equation}
with some $\xi_i \in L_0\left[-\frac{D}{q}, \frac{1}{2}\mathbb{Z}\right]$. This in particular implies that $\xi_i \in L_0^{*}$ for all $i$.\\\\
We now assert we can take
\begin{equation*}
    \xi_i \in \mathcal{P}_{S} = \{y'=(y_1,y,y_2)\in \mathbb{R}^{n+2} \mid y_1>0, \ \phi_0[y']>0\}    
\end{equation*}
of Section \ref{preliminaries} for all $i$. The second condition is clear, as $S_0[\xi_i]=-2D/q>0$ because we take $D<0$. For the first one, we can always multiply with $\textup{diag}(-1, 1_n, -1) \in \Gamma(L_0)$ and the assertion follows.\\\\
Now, if $\xi \in \Xi_d$ of Lemma \ref{equivalence_classes}, we can write $\xi = \gamma \xi_j$ for some $1 \leq j \leq h$ and $\gamma \in \Gamma(L_0)$. But, $\xi \in \mathcal{P}_S$ as well, so we must have $\gamma \in \Gamma(L_0) \cap G^{*,0}_{\mathbb{R}} = \Gamma^{+}(L_0)$. Indeed, if $\gamma \in G_{\mathbb{R}}^{*}\backslash G_{\mathbb{R}}^{*,0}$, then $\widetilde{\gamma}:=\textup{diag}(1, \gamma, 1) \in G_{\mathbb{R}}\backslash G_{\mathbb{R}}^{0}$. But then $\widetilde{\gamma}\langle i\xi_j\rangle = i\xi$ and $i\xi,\ i\xi_j \in \mathcal{H}_S$. This is a contradiction, because $\widetilde{\gamma}$ sends $\mathcal{H}_S$ to $-\mathcal{H}_S := \{z=x-iy \in V_0 \otimes_{\mathbb{R}} \mathbb{C} \mid y \in \mathcal{P}_S\}$ (see \cite[p. 18]{eisenstein_thesis}).\\

Therefore, from \cite[p. 26]{eisenstein_thesis}, we have $A(\xi) = A(\gamma \xi_i) = A(\xi_i)$. Define now
\begin{multline*}
    n(\xi_i;d) := \#\left\{s \in \mathbb{Z}^{n}/dS\mathbb{Z}^{n} \mid D \equiv \frac{1}{2}qS^{-1}[s] \pmod{qd}, \,\, \left(\frac{\frac{1}{2}qS^{-1}[s]-D}{qd}, S^{-1}s, d\right)^{t} = \gamma \xi_i, \right.\\\left.\gamma \in \Gamma^{+}(L_0)\right\}.
\end{multline*}
From the above considerations, Lemma \ref{equivalence_classes} and Proposition \ref{inner_product}, we may write, for $(N,D)=1$
\[
\langle V_N^* \phi_N, P_{k,D,r} \rangle = \sum_{i=1}^h \sum_{0< d|N} d^{k-(n+1)}n(\xi_i; d) A\left( \frac{N}{d} \xi_i\right).
\]
In particular, we arrive at the following Proposition.
\begin{proposition}\label{first form of result}
    Let $(D,r) \in \widetilde{\textup{supp}}(L, \sigma)$. Let $\mathcal{P}$ be a finite set of primes, which includes the prime factors of $D$. Let
    \begin{equation*}
    \mathcal{D}_{F,\mathcal{P}_{k,D,r},\mathcal{P}}(s) := \sum_{\substack{N=1\\(N,p)=1\forall p \in \mathcal{P}}}^{\infty} \langle V_N^* \phi_N, P_{k,D,r}\rangle N^{-s},
    \end{equation*}
    which converges absolutely for $\textup{Re}(s)>k+1$ by comparison with $D_{F,\mathcal{P}_{k,D,r}}(s)$ (see Lemma \ref{convergence_dirichlet}). Let also
    \[
    \zeta_{\xi_i, \mathcal{P}}(s):= \sum_{\substack{N=1\\(N,p)=1 \forall p\in \mathcal{P}}}^{\infty} n(\xi_i; N)N^{-s}.
    \]
We then have that
\begin{equation}\label{expression2}
\mathcal{D}_{F,\mathcal{P}_{k,D,r},\mathcal{P}}(s) = \sum_{i=1}^h \zeta_{\xi_i, \mathcal{P}}(s-k+n+1) D_{F,\xi_i, \mathcal{P}}(s),
\end{equation}
where $D_{F,\xi_i,\mathcal{P}}(s) := \sum_{\substack{N=1\\(N,p)=1 \forall p\in \mathcal{P}}}^{\infty} A(N \xi_i) N^{-s}$.
\end{proposition}
\section{Relation to Sugano's Theorem}\label{relation to sugano}
In order to now obtain an Euler product, we will make use of the main theorem of Sugano in his paper \cite{sugano}. We first need some setup.\\

For each prime number $p < \infty$, we define $K_{p} := G_{p} \cap \textup{SL}_{n+4}(\mathbb{Z}_{p}) = G(\mathbb{Z}_p)$ and let $K_{f} := \prod_{p < \infty} K_{p}$.
\begin{flushleft}
    We remind ourselves here that $G_{\infty}^{0}$ acts transitively on $\mathcal{H}_{S}$ (cf. Section \ref{preliminaries}). Let $\mathcal{Z}_0$ denote any point of $\mathcal{H}_{S}$ with real part $0$ and denote by $K_{\infty}$ its stabiliser in $G_{\infty}^{0}$. Then, we have that $G_{\infty}^{0}/K_{\infty} \cong \mathcal{H}_{S}$.
\end{flushleft}
\begin{defn}
    Let $k \geq 0$. A function $\bm{F} : G_{\mathbb{A}} \longrightarrow \mathbb{C}$ is called a holomorphic cusp form of weight $k$ with respect to $K_{f}$ if the following conditions hold:
    \begin{enumerate}
        \item $\bm{F}(\gamma g u) = \bm{F}(g) \textup{ }\forall \gamma \in G_{\mathbb{Q}}, u \in K_{f}$.
        \item For any $g=g_{\infty}g_{f}$ with $g_{\infty} \in G_{\infty}^{0}$ and $g_{f} \in G_{\mathbb{A},f}$, $\bm{F}(g_{\infty}g_{f})j(g_{\infty},\mathcal{Z}_0)^{k}$ depends only on $g_{f}$ and $\mathcal{Z} = g_{\infty}\langle \mathcal{Z}_0\rangle$ and is holomorphic on $\mathcal{H}_{S}$ as a function of $\mathcal{Z}$.
        \item $\bm{F}$ is bounded on $G_{\mathbb{A}}$.
    \end{enumerate}
    
    Denote the above space by $\mathfrak{S}_{k}(K_{f})$.
\end{defn}
For each $g_{f} \in G_{\mathbb{A},f}$ and $\mathcal{Z} \in \mathcal{H}_{S}$, we define
\begin{equation*}
    \bm{F}(g_{f}; \mathcal{Z}) := \bm{F}(g_{\infty}g_{f})j(g_{\infty}, \mathcal{Z}_0)^{k},
\end{equation*}
where $g_{\infty} \in G^{0}_{\infty}$ is chosen so that $\mathcal{Z} = g_{\infty}\langle \mathcal{Z}_0\rangle$. Let now
\begin{equation*}
    \Gamma(g_{f}) = G_{\mathbb{Q}} \cap (G_{\infty}^{0} \times g_{f}K_{f}g_{f}^{-1}),
\end{equation*}
which is a discrete subgroup of $G_{\infty}^{0}$. We then have
\begin{equation*}
    \bm{F}(g_{f}; \gamma \langle \mathcal{Z}\rangle) = j(\gamma, \mathcal{Z})^{k}\bm{F}(g_{f}; \mathcal{Z})
\end{equation*}
for all $\gamma \in \Gamma(g_{f})$ and $\mathcal{Z} \in \mathcal{H}_S$. Now, if $X \in V_0$, define the element $\gamma_{X} \in G_{\mathbb{Q}}$ by
\begin{equation*}
    \gamma_{X} := \m{1&-X^{t}S_0 & -\frac{1}{2}S_0[X]\\ 0&1_{n+2}&X\\0&0&1}.
\end{equation*}
Now, the holomorphic function $\bm{F}(g_{f};\mathcal{Z})$ is invariant under $\mathcal{Z} \longmapsto \mathcal{Z}+X$ for $X$ in the lattice
\begin{equation*}
    L(g_{f}):=\{X \in V_{0} \mid \gamma_{X} \in \Gamma(g_f)\}.
\end{equation*}
Hence, every such function admits a Fourier expansion of the form
\begin{equation*}
    \bm{F}(g_{f}; \mathcal{Z}) = \sum_{\substack{r \in \hat{L}(g_{f})\\ir \in \mathcal{H}_S}}a(g_{f}; r)e(2\phi_0(r, Z)),
\end{equation*}
where $\hat{L}(g_{f}):=\left\{X \in V_0 \mid 2\phi_0(X,Y) \in \mathbb{Z} \textup{ for all } Y\in L(g_f)\right\}$ is the dual lattice of $L(g_f)$.\\\\
Finally, let us introduce adelic Fourier coefficients. Let $\chi = \prod_{p \leq \infty}\chi_p$ be a character of $\mathbb{Q}_{\mathbb{A}}$ such that $\chi_{\mid \mathbb{Q}}=1$ and $\chi_{\infty}(x) = e(x)$ for all $x \in \mathbb{R}$. For $\eta \in V_0$ and $g \in G_{\mathbb{A}}$, we define
\begin{equation*}
    \bm{F}_{\chi}(g;\eta) := \int_{V_0\backslash V_{\mathbb{A}}} \bm{F}(\gamma_{\chi}g)\chi\left(-2\phi_0(\eta,X)\right) dX.
\end{equation*}
Now, for $g_{\infty} \in G_{\infty}^{0}$ and $g_{f} \in G_{\mathbb{A},f}$ we have
\begin{equation}\label{property of F_x}
    \bm{F}_{\chi}(g_{\infty}g_{f}; \eta) = a(g_{f};\eta)j(g_{\infty};\mathcal{Z}_0)^{-k}e(2\phi_0(\eta, g_{\infty}\langle \mathcal{Z}_0\rangle)).
\end{equation}
We then have from the above definitions (see \cite[(1.14)]{sugano}):
\begin{itemize}
    \item $\bm{F}_{\chi}(\gamma_Xgu; \eta) = \chi(2\phi_0(\eta, X))\bm{F}_{\chi}(g;\eta) \textup{ for all } X \in V_{\mathbb{A}}, u \in K_{f}$,
    \item $\bm{F}_{\chi}\left(\m{\alpha&&\\&\beta&\\&&\alpha^{-1}} g;\eta\right) = \bm{F}_{\chi}(g;\beta^{-1}\eta\alpha), \textup{ for all } \alpha \in \mathbb{Q}^{\times}, \beta \in G^{*}_{\mathbb{Q}}$.
    \item $\bm{F}(\gamma_{X}g) = \sum_{\eta \in V_0} \bm{F}_{\chi}(g;\eta)\chi(2\phi_0(\eta, X)) \textup{ for all } X \in V_{0,\mathbb{A}}$. 
\end{itemize}
We now want to show that there is a bijection between the spaces $\mathfrak{S}(K_{f})$ and $S_k(\Gamma_S)$. This is true because of the following Lemma.
\begin{lemma}\label{correspondence automorphic forms}
    We have $G_{\mathbb{A}} = G_{\mathbb{Q}}G_{\infty}^{0}K_f$.
\end{lemma}
\begin{proof}
    From the proof of \cite[Lemma 1]{sugano}, we have $G_{\mathbb{A}} = G_{\mathbb{Q}}G_{\mathbb{A},f}^{*}G_{\infty}^0K_{f}$, where we view $G^{*}$ as a subgroup of $G$ via $g \longmapsto \textup{diag}(1,g,1)$. But, from Lemma \ref{class_number_g}, we have $G_{\mathbb{A}}^{*} = G_{\mathbb{Q}}^{*}G_{\infty}^{*}K_{f}^{*}$, where we now define
    \begin{equation*}
        K_{f}^{*} := \prod_{p < \infty} G^{*}(\mathbb{Z}_p).
    \end{equation*}
    From this, we obtain $G_{\mathbb{A},f}^{*} = G_{\mathbb{Q}}^{*}K_{f}^{*}$. Therefore, because elements of $K_{f}^{*}$ and $G_{\infty}^{0}$ commute, we obtain
    \begin{equation*}
        G_{\mathbb{A}} = G_{\mathbb{Q}}G_{\mathbb{Q}}^{*}G_{\infty}^{0}K_{f}^{*}K_{f} = G_{\mathbb{Q}}G_{\infty}^{0}K_f,
    \end{equation*}
    as required.
\end{proof}
The bijection is now given by $\bm{F} \longmapsto F(\mathcal{Z}):= \bm{F}(g_{\infty})j(g_{\infty}, \mathcal{Z}_0)^{k}$, where $g_{\infty} \in G_{\infty}^{0}$ is chosen so that $g_{\infty} \langle \mathcal{Z}_0 \rangle = \mathcal{Z}$.\\

 
Let now $g_{f}=(id, id, \cdots)$, which we denote by just $id$. It then follows that $F(\mathcal{Z}) = \bm{F}(id, \mathcal{Z}), \,\, \forall \mathcal{Z} \in \mathcal{H}_S$. In that case, we have $L(g_{f}) = L_0 = \mathbb{Z}^{n+2}$. Now, $a(id, r)=A(r)$ for all $r \in L_0^{*}$. Hence
\begin{equation*}
    a(id, \xi_i) = A(\xi_{i}),
\end{equation*}
for all $i$, where $\xi_{i}$'s are the representatives for $L_0\left[-\frac{D}{q}, \frac{1}{2}\mathbb{Z}\right]/\Gamma(L_0)$, as in \eqref{xi_i reps}.\\\\
Fix now a complete system of representatives $\{u_i\}_{i=1}^{h}$ for $H(\xi)_{\mathbb{Q}}\backslash H(\xi)_{\mathbb{A}}/(H(\xi)_{\mathbb{A}}\cap C)$, corresponding to the $\xi_i's$ of \eqref{xi_i reps}, in the sense of Proposition \ref{genus-classes}. Assume also that $u_{i,\infty}=1$ for all $i=1,\cdots, h$. We then prove the following Lemma.
\begin{lemma}\label{u_i, xi_i}
    Let $\xi$ be as we have specified it right after Lemma \ref{class_number_g}. If now $\xi_i$'s and $u_i$'s are as above, we have
    \begin{equation*}
        a(u_i, \xi) = a(id, \xi_{i}),
    \end{equation*}
    for all $i=1, \cdots, h$. Therefore, $a(u_i, \xi) = A(\xi_i)$ for all $i=1, \cdots, h$.
\end{lemma}
\begin{proof}
    By the definition of the $\xi_{i}$'s and Witt's theorem (\cite[Lemma 1.5 (ii)]{shimura2004arithmetic}), there is $\alpha \in G_{\mathbb{Q}}^{*}$ such that $\xi_i = \alpha \xi$. By the correspondence given by Shimura in Proposition \ref{genus-classes}, we get $\alpha^{-1}L_0 = u_iL_0$. This then implies $\alpha u_i L_0 = L_0$, so $\alpha u_i \in C$ (we remind ourselves that $C = \{x \in G^{*}_{\mathbb{A}} \mid xL_0 = L_0\}$). By definition, we then get $\alpha u_i \in K_{f}$. Hence, we obtain
    \begin{equation*}
        \bm{F}_{\chi}(u_i; \xi) = \bm{F}_{\chi}(u_i; \alpha^{-1}\xi_{i}) = \bm{F}_{\chi}\left(\textup{diag}(1, \alpha, 1)u_i; \xi_i\right) = \bm{F}_{\chi}\left(\left(\textup{diag}(1, \alpha, 1),id\right); \xi_i\right),
    \end{equation*}
    where $\left(\textup{diag}(1, \alpha, 1),id\right)$ denotes the element of $G_{\mathbb{A}}$ with infinity part $\textup{diag}(1, \alpha, 1)$. The second equality above follows from the second bullet point and the third equality from the first bullet point, just before Lemma \ref{correspondence automorphic forms}. By the property of $\bm{F}_{\chi}$ in \eqref{property of F_x}, we obtain
    \begin{equation*}
        a(u_i;\xi)e\left(2\phi_0(\xi, \mathcal{Z}_0)\right) = a(id, \xi_i)j\left(\textup{diag}(1, \alpha, 1), \mathcal{Z}_0\right)^{-k}e\left(2\phi_0\left(\xi_i, \textup{diag}(1, \alpha, 1)\langle \mathcal{Z}_0 \rangle\right)\right).
    \end{equation*}
    But $j\left(\textup{diag}(1, \alpha, 1), \mathcal{Z}_0\right)=1$ and $\textup{diag}(1, \alpha, 1)\langle \mathcal{Z}_0 \rangle = \alpha \mathcal{Z}_0$ which then gives that the right hand side above equals
    \begin{equation*}
        a(id;\xi)e\left(2\phi_0(\xi_i, \alpha \mathcal{Z}_0)\right) = a(id;\xi)e\left(2\phi_0(\alpha\xi,\alpha \mathcal{Z}_0)\right) = a(id;\xi)e\left(2\phi_0(\xi, \mathcal{Z}_0)\right),
    \end{equation*}
    because $\xi_i = \alpha \xi$. This then gives the result.
\end{proof}
For each prime $p$ and $g_f \in G_{\mathbb{A},f}$, let 
\begin{equation*}
    M(g_f;\xi)_p := H(\xi)_p \cap g_{f}K_{f}g_{f}^{-1} \textup{ and } M(g_{f};\xi)_{f} := \prod_{p}M(g_f;\xi)_p.
\end{equation*}
Define then $e(\xi)_i := \#\{H(\xi)_{\mathbb{Q}}\cap M(u_ig_{f};\xi)_{f}\}$ for $1\leq i\leq h$ and $\displaystyle{\mu(\xi) := \sum_{i=1}^{h}e(\xi)^{-1}_{i}}$.

Let also $V(g_f;\xi)$ the space of $\mathbb{C}$-valued functions on $H(\xi)_{\mathbb{A}}$, which are left $H(\xi)_{\mathbb{Q}}$ and right $H(\xi)_{\infty}M(g_f;\xi)_{f}$ invariant. 
We now have the following Theorem, which follows from \cite[Theorem 1]{sugano}.
\begin{theorem}\label{sugano}
    Assume $F \in S_k(\Gamma_{S})$ corresponds to $\bm{F} \in \mathfrak{S}_{k}(K_{f})$, as above. Assume $\bm{F}$ is a simultaneous eigenfunction of the Hecke pairs $\mathcal{H}_p = (G_p, K_p)$ for all $p$ (see \cite[Section 2]{sugano}). We also assume $A(\xi) \neq 0$ and $g_f=id$. Then, there is a finite set of primes $\mathcal{P}$, such that if $f \in V(g_{f};\xi)$ is a simultaneous eigenfunction of the Hecke algebras $\mathcal{H}_{p}' := (H(\xi)_{p}, M(g_{f};\xi)_{p})$ for all $p \notin \mathcal{P}$, we have
    \begin{multline*}
    \sum_{\substack{N=1\\(N,p)=1 \forall p\in \mathcal{P}}}^{\infty} \mu(\xi)^{-1}\sum_{i=1}^{h} A(N\xi_i)\frac{\overline{f}(u_i)}{e(\xi)_i}N^{-(s+k-\frac{n+2}{2})} = \left(\mu(\xi)^{-1}\sum_{i=1}^{h} A(\xi_i)\frac{\overline{f}(u_i)}{e(\xi)_i}\right)\times\\\times L_{\mathcal{P}}(F;s)
    L_{\mathcal{P}}\left(\overline{f};s+1/2\right)^{-1}\times(\textup{zeta})^{-1}(s),
    \end{multline*}
    where $(\textup{zeta})(s) := \begin{cases}1 & \textup{if }n \textup{ odd} \\ \zeta_{\mathcal{P}}(2s) & \textup{if } n \textup{ even}
    \end{cases}$.\\\\
    Here, $L(-,s)$ denotes the standard $L$-function attached to orthogonal modular forms, as this is defined in \cite[Paragraph 4-1]{sugano}. Also, for any $L$-function, we write $L_{\mathcal{P}}$ for the Euler product not containing the primes in $\mathcal{P}$.
\end{theorem}
\begin{proof}
    This follows by Sugano's main Theorem in \cite[Theorem 1]{sugano}. In our setting, we take $g_{f} = (id, id, \cdots)$ and then substitute $a(u_i;N\xi)$ with $A(N\xi_i)$ in the original form of \cite[Theorem 1]{sugano}, due to Lemma \ref{u_i, xi_i}. We also note that in this case $H(\xi)_{\infty}M(g_{f};\xi)_{f} = H(\xi)_{\mathbb{A}} \cap C$, where $C$ is defined in \eqref{stabiliser}. The set of primes $\mathcal{P}$ contains all the primes contained in the set $\mathcal{P}_2$ of \cite[Theorem 1]{sugano} and the finite set of primes $p$ for which $\partial_p \neq 0$ or $\partial'_p \neq 0$, where $\partial_p, \partial_p'$ are defined in \cite[Theorem 1]{sugano}. We note that in our case, $L_p$ is maximal for all $p$, so $\mathcal{P}_1$ in \cite[Theorem 1]{sugano} is empty.
\end{proof}
From now on, we fix $g_f=id$. For any $f \in V(g_f; \xi)$, we set
\begin{equation}\label{a_f}
    \widetilde{f}(u_i) := \frac{\overline{f}(u_i)}{e(\xi)_i}\mu(\xi)^{-1} \textup{, }i=1,\cdots, h  \,\ \textup{and } \displaystyle{A_f := \sum_{i=1}^{h}\widetilde{f}(u_i)A(\xi_i)}.
\end{equation}
The formula in Theorem \ref{sugano} then becomes
\begin{equation*}
    (\textup{zeta})(s) \times L_{\mathcal{P}}\left(\overline{f}; s+1/2\right)\sum_{i=1}^{h}\widetilde{f}(u_i)D_{F,\xi_i,\mathcal{P}}\left(s+k-(n+2)/2\right) = A_fL_{\mathcal{P}}(F;s),
\end{equation*}
where $D_{F,\xi_i,\mathcal{P}}(s)$ is the Dirichlet series appearing in Proposition \ref{first form of result}.
This is true for any simultaneous eigenfunction $f$ (hence $\widetilde{f}$ as well) of the Hecke algebras $H_{p}' = (H(\xi)_p, M(g_f; \xi)_p)$ for all $p\notin \mathcal{P}$. Our aim is to invert it so that we solve for $D_{F, \xi_i, \mathcal{P}}(s)$.
\newline\newline
From the definition of the $u_i's$ just before Lemma \ref{u_i, xi_i}, we have that 
\begin{equation*}
H(\xi)_{\mathbb{A}} = \bigsqcup_{i=1}^{h}H(\xi)_{\mathbb{Q}}u_iD,
\end{equation*}
where $D := H(\xi)_{\infty}M(g_{f};\xi)_{f}$. We note $D$ is an open subgroup of $H(\xi)_{\mathbb{A}}$ and $D \cap H(\xi)_{f}$ is compact. By \cite[Lemma 17.6]{shimura2004arithmetic}, there is a correspondence
\begin{equation*}
    f \longleftrightarrow \{f^{(i)} \ |\  i=1,\cdots, h\},
\end{equation*}
with each $f^{(i)} \in \mathbb{C}$, because $H(\xi)_{\infty} \cong \hbox{SO}(n+1)$ is compact. We also note here that $f(u_i) = f^{(i)}$ for all $i=1,\cdots, h$, as we can see by the way these $f^{(i)}$ are defined in the proof of \cite[Lemma 10.8]{Shimura_Euler_Product}.
\newline\newline
Now, for any two simultaneous eigenfunctions $f_{i}, f_j$ of the Hecke algebras defined by the pairs $\mathcal{H}_{p}' = (H(\xi)_p, M(g_f, \xi)_p)$ for all $p \notin \mathcal{P}$, their inner product is defined via the formula
\begin{equation*}
    \langle f_i, f_j \rangle = \left\{\sum_{k=1}^{h}\nu\left(\Gamma^{k}\right)\right\}^{-1}\sum_{k=1}^{h}\nu\left(\Gamma^{k}\right)\overline{f_i^{(k)}}f_{j}^{(k)} = \left\{\sum_{k=1}^{h}\nu\left(\Gamma^{k}\right)\right\}^{-1}\sum_{k=1}^{h}\nu\left(\Gamma^{k}\right)\overline{f_i(u_k)}f_{j}(u_k),
\end{equation*}
where $\Gamma^{k} = H(\xi)_{\mathbb{Q}} \cap u_kDu_k^{-1}$ and $\nu\left(\Gamma^{k}\right) = \#\left(\Gamma^{k}\right)^{-1}$, as in \cite[(17.23)]{shimura2004arithmetic} (here we again use the fact that $H(\xi)_{\infty}$ is compact). \\

As $e(\xi)_i = \#\{H(\xi)_{\mathbb{Q}}\cap M(u_ig_f;\xi)_f\}$, we have that $e(\xi)_{i} = \nu\left(\Gamma^{i}\right)^{-1}$, which also gives 
\begin{equation*}
    \mu(\xi) = \sum_{k=1}^{h}\nu\left(\Gamma^{k}\right).
\end{equation*}
It is now possible to choose a basis of orthonormal Hecke eigenforms $\{f_1,f_2, \cdots, f_h\}$ for $V(g_f; \xi)$ with respect to the above inner product (i.e. $\langle f_i,f_j\rangle=\delta_{ij}$ for all $i,j$). This is true because the Hecke algebra defined by $\mathcal{H}_{p}'$ is commutative and consists of self-adjoint operators for all $p \notin \mathcal{P}$ (see proof of \cite[Proposition 17.14]{shimura2004arithmetic}). Also, by \cite[Lemma 17.6, (1)]{shimura2004arithmetic}, there is an isomorphism between $V(g_f; \xi)$ and $\mathbb{C}^{h}$. Therefore, this basis must consist of $h$ eigenforms. Hence, we get the expression
\begin{multline*}
    D_{F,\xi_i,\mathcal{P}}\left(s+k-(n+2)/2\right) = \mu(\xi)^{-1}(\textup{zeta})^{-1}(s)L_{\mathcal{P}}(F;s)\sum_{j=1}^{h}\nu\left(\Gamma^{i}\right)\frac{\mu(\xi)}{e(\xi)_i^{-1}}f_{j}(u_i)\times\\\times L_{\mathcal{P}}\left(\overline{f_j};s+1/2\right)^{-1}A_{f_j},
\end{multline*}
which after the simplifications becomes
\begin{multline}\label{expression}
D_{F, \xi_i,\mathcal{P}}(s) = (\textup{zeta})^{-1}\left(s-k+(n+2)/2\right) L_{\mathcal{P}}\left(F; s-k+(n+2)/2\right)\times\\\times\sum_{j=1}^{h}f_j(u_i)L_{\mathcal{P}}\left(\overline{f_j};s-k+(n+3)/2\right)^{-1}A_{f_{j}}.
\end{multline}
Hence, we arrive at the following Theorem.
\begin{theorem}\label{dirichlet}
    Let $(D,r) \in \widetilde{\textup{supp}}(L, \sigma)$. Let $\mathcal{P}$ be a finite set of primes, containing the primes described in the proof of Theorem \ref{sugano} and all the prime divisors of $D$. Let $F \in S_{k}(\Gamma_S)$ corresponding to $\bm{F} \in \mathfrak{S}(K_f)$ with $A(\xi) \neq 0$. Assume $\bm{F}$ is a simultaneous eigenfunction for the Hecke algebra $\mathcal{H}_p$, defined by the pair $(G_p, K_p)$ for all $p$ and let $\mathcal{P}_{k,D,r}$ denote the Poincar\'e series of \eqref{poincare in maass}. Let also $\{f_{j}\}_{j=1}^{h} \in V(g_f;\xi)$ denote an orthonormal basis of simultaneous eigenfunctions for the pairs $\mathcal{H}_p' = (H(\xi)_p, M(g_f;\xi)_p)$ for all $p \notin \mathcal{P}$, $A_{f_j}$ as in \eqref{a_f}, and denote with $L_{\mathcal{P}}(-, s)$ the standard $L$-function attached to either $F$ or any $f_{j}$, by ignoring the $p$-factors for $p \in \mathcal{P}$. We then have
    \begin{equation*}
        \mathcal{D}_{F, \mathcal{P}_{k,D,r}, \mathcal{P}}(s) = \sum_{\substack{N=1\\(N,p)=1\forall p \in \mathcal{P}}}^{\infty} \langle V_N^* \phi_N, P_{k,D,r}\rangle N^{-s} = 
    \end{equation*}
    \begin{equation*}
    =L_{\mathcal{P}}\left(F;s-k+(n+2)/2\right)\sum_{j=1}^{h}A_{f_j}L_{\mathcal{P}}\left(\overline{f_j};s-k+(n+3)/2\right)^{-1}\sum_{i=1}^{h}\zeta_{\xi_i,\mathcal{P}}(s-k+n+1)f_{j}(u_i)\times
\end{equation*}
\begin{equation*}
    \times \begin{cases}1 & \textup{if }n \textup{ odd} \\ \zeta_{\mathcal{P}}(2s-2k+n+2)^{-1} & \textup{if } n \textup{ even}
    \end{cases},
\end{equation*}
where $\zeta_{\xi_{i}, \mathcal{P}}(s)$ are as in Proposition \ref{first form of result}.
\end{theorem}
\begin{proof}
    By substituting the expression we deduced in \eqref{expression} into \eqref{expression2}, we obtain (here we denote by "$\textup{zeta}$" the function of Theorem \ref{sugano} after $s \longmapsto s-k+(n+2)/2$)
\begin{equation*}
    \sum_{\substack{N=1\\(N,p)=1\forall p \in \mathcal{P}}}^{\infty} \langle V_N^* \phi_N, P_{k,D,r}\rangle N^{-s} = \sum_{i=1}^h \zeta_{\xi_i,\mathcal{P}}(s-k+n+1) D_{F,\xi_i,\mathcal{P}}(s) =
\end{equation*}
\begin{equation*}
    = (\textup{zeta})^{-1}\times\sum_{i=1}^{h}\zeta_{\xi_i,\mathcal{P}}(s-k+n+1) L_{\mathcal{P}}\left(F;s-k+\frac{n+2}{2}\right)\sum_{j=1}^{h}f_{j}(u_i)L_{\mathcal{P}}\left(\overline{f_j};s-k+\frac{n+3}{2}\right)^{-1}A_{f_j}
\end{equation*}
\begin{equation*}
    =(\textup{zeta})^{-1}\times L_{\mathcal{P}}\left(F;s-k+\frac{n+2}{2}\right)\sum_{j=1}^{h}A_{f_j}L_{\mathcal{P}}\left(\overline{f_j};s-k+\frac{n+3}{2}\right)^{-1}\sum_{i=1}^{h}\zeta_{\xi_i,\mathcal{P}}(s-k+n+1)f_{j}(u_i). \qedhere
\end{equation*}
\end{proof}
Hence, we would like to explore the connection between $\sum_{i=1}^{h}\zeta_{\xi_i}(s-k+n+1)f_{j}(u_i)$ and $L\left(\overline{f_j};s-k+(n+3)/2\right)$.\\

It turns out we can now obtain a clear-cut Euler product expression in the case $h=1$ and when $D$ is a specifically chosen number. The question is how we can pick $S$, so that we can get $h=1$. This is the theme of the next Section.


\section{Explicit Examples}\label{explicit examples}
We will now focus our attention to some specific examples of matrices $S$ and corresponding Poincar\'e series. In particular, we set $D=-q$ and choose matrices $S$, so that the number of representatives for $L_0\left[1, \frac{1}{2}\mathbb{Z}\right]/\Gamma(L_0)$ is $1$. We therefore take $G$ to be the Poincar\'e series $\mathcal{P}_{k,-q,r}$, with $r \in L$ such that $(-q,r) \in \widetilde{\textup{supp}}(L,\sigma)$. In particular, those choices imply that we can take
\begin{equation*}
    \xi = (1,0,\cdots, 0,1)^{t}
\end{equation*}
as an element of $L_0\left[1, \frac{1}{2}\mathbb{Z}\right]$.
Therefore, $\zeta_{\xi,\mathcal{P}}(s)$ of Proposition \ref{first form of result} can be written as
    \begin{equation*}
        \zeta_{\xi,\mathcal{P}}(s) = \sum_{\substack{N=1\\(N,p)=1 \forall p\in \mathcal{P}}}^{\infty} n(\xi; N)N^{-s},
    \end{equation*}
    where this time
    \begin{equation*}
        n(\xi;d) = \#\left\{s \in \mathbb{Z}^{n}/dS\mathbb{Z}^{n} \mid D \equiv \frac{1}{2}qS^{-1}[s] \pmod{qd}\right\}.
    \end{equation*}
In these cases, we are able to deduce an exact Euler product expression, connecting the Dirichlet series of interest and the standard $L$-function of the orthogonal group.  
\subsection{Examples with rank $1$}\label{examples with rank 1}
We consider the case where $S = 2t$ for some $t \geq 1$ with $t$ square-free. This condition is needed so that the lattice $L = \mathbb{Z}$ (and therefore $L_0$ and $L_1$) is maximal (see \cite[Example 1.6.6(ii)]{eisenstein_thesis}). Now $V_0=\mathbb{Q}^{3}$ and the quadratic form of interest is then
\begin{equation*}
    \phi_0(x,y) = \frac{1}{2}x^{t}S_0y,
\end{equation*}
for $x,y \in V_0$, where $S_0 = \m{&&1\\&-2t&\\1&&}$. Hence, $\phi_0$ is represented by $S_0/2$ with respect to the standard basis $e_1,e_2, e_3$ of $V$. \\\\
By \cite[Paragraph 7.3]{shimura2004arithmetic} we have that there exists a quaternion algebra $B$ over $\mathbb{Q}$ such that we can put $V_0 = B^{\circ}\zeta$, $\phi_0[x\zeta] = dxx^{\iota}$, $2\phi_0(x\zeta, y\zeta) = d\textup{Tr}_{B/\mathbb{Q}}(xy^{\iota})$, where $B^{\circ}=\{x\in B \mid x^{\iota}=-x\}$ with $\iota$ the main involution of $B$ and $\zeta \in A(V_0)$ such that $\zeta^{2} = -d$. Here, by $A(V_0)$ we mean the Clifford algebra of $(V_0,\phi_0)$ (see \cite[Chapter 2]{shimura2004arithmetic}). Also, in general, we define $\textup{Tr}_{B/\mathbb{Q}}(x) := x+x^{\iota}$ for $x \in B$. \\\\
Now, from \cite[Paragraph 7.3]{shimura2004arithmetic}, we have a way to compute $\zeta$ and $d$. We first need a basis $h_1,h_2,h_3$ of $V$ such that $\phi_0(h_i,h_j)=c_{i}\delta_{ij}$ for all $1\leq i,j\leq 3$. In our case, we can make a choice
\begin{equation*}
    h_1 = e_1+e_3, h_2 = e_2, h_3 = e_1-e_3.
\end{equation*}
Then, we get that the condition above is satisfied with $c_1 = 1, c_2 = -t, c_3 = -1$ and therefore $d=c_1c_2c_3 = t$. Also, even though it is not needed in what follows, $\zeta=h_1h_2h_3$. \\\\
We now remind ourselves that $\xi = (1,0,1)^{t}$ and $W=(\mathbb{Q}\xi)^{\perp}$. This then implies that $\phi_0[\xi]=1$. From \cite[Paragraph 5.2]{shimura_diophantine_2006}, we get that $\exists k \in B^{\circ}$ such that $\xi = k\zeta$. Then, if $K:= \mathbb{Q}+\mathbb{Q}k$, we get that $K=\mathbb{Q}(\sqrt{-t})$ because $-t$, which is $-d\phi_0[\xi]$, cannot be a square in $\mathbb{Q}^{\times}$.\\\\
Using now \cite[Theorem 5.7]{shimura_diophantine_2006} and the formula $(5.11)$ given there, tailored to our situation, we have the following Theorem.
\begin{theorem}\label{rank 1 main theorem}
    We define the following quantities:
    \begin{itemize}
        \item Let $K:=\mathbb{Q}(\sqrt{-t})$ and $c_K$ denote the class number of $K$.
        \item $\mathfrak{c}$ denotes the ideal of $\mathbb{Z}$, determined by the local conditions
            \begin{equation}\label{conductor}
                \mathfrak{c}_p^{2}N_{K/\mathbb{Q}}(\mathfrak{d}_{K/\mathbb{Q}})_p = \mathfrak{a}_p\phi_0[\xi]\phi_0(\xi, L_0)_{p}^{-2},
            \end{equation}
            for all primes $p$, where $\mathfrak{a} = t\mathbb{Z}$ and $\mathfrak{d}_{K/\mathbb{Q}}$ is the different ideal.
        \item For a prime $p$ dividing $\mathfrak{c}$, we define $[K/\mathbb{Q}, \ p]$ to be $-1, 0 \textup{ or 1}$, according to whether $p$ remains prime, ramifies or splits in $K$.
        \item Let $p$ be a rational prime. Pick $\epsilon_p \in -\det(\phi_{0,p})(\mathbb{Q}_p^{\times})^{2}$, which is either a unit or a prime element of $\mathbb{Q}_p$ and choose an element $\beta_p \in \mathbb{Q}_p$ such that $\phi_0(\xi, L_{0,p}) = \beta_p\mathbb{Z}_p$. Define then $r_p(\xi) := \epsilon_p^{-1}\phi_0[\xi]\beta_p^{-2}$. Define also
        \begin{equation*}
            \mathfrak{C}_p:= \{u^2+4w \mid u, w \in \mathbb{Z}_p\}.
        \end{equation*}
        \item $\mathfrak{a}^{*}$ is the product of the prime factors $p$ of $t$ such that $r_{p}(\xi) \in p^{-1}\mathbb{Z}_p$ and $r_p(\xi) \notin \mathfrak{C}_p$. 
        \item $\mu$ is the number of prime ideals dividing $\mathfrak{a}^{*}$ and ramified in $K$.
        \item $U := \mathcal{O}_K^{\times}$ and $U' := \{x \in \mathcal{O}_K^{\times} \mid x-1 \in \mathfrak{c}_p(\mathfrak{d}_{K/\mathbb{Q}})_{p} \ \forall p \nmid \mathfrak{a}^{*}\}$.
    \end{itemize} 
    We then have
    \begin{equation}\label{index}
        [H(\xi)_{\mathbb{A}}:H(\xi)_{\mathbb{Q}}(H(\xi)_{\mathbb{A}}\cap C)] = c_K 2^{1-\mu}[U:U']^{-1}N(\mathfrak{c}) \prod_{p \mid \mathfrak{c}}\left\{1-\frac{1}{p}[K/\mathbb{Q}, \ p]\right\},
    \end{equation}
    \begin{proof}
        This follows from \cite[Theorem 5.7]{shimura_diophantine_2006}. In our case, we have, in the notation of the Theorem:
        \begin{itemize}
            \item The base field $F$ is $\mathbb{Q}$, which has class number $1$.
            \item The product of all the prime ideals in $\mathbb{Q}$ for which $B$ ramifies, $\mathfrak{e}$, equals $\mathbb{Z}$. This is true because for each prime $p$, $\phi_0$, viewed as a bilinear form over $\mathbb{Q}_{p}$, is isotropic (the Witt index is $1$ for all primes $p$). Therefore by \cite[Paragraph 7.3]{shimura2004arithmetic}, $B$ over $\mathbb{Q}_p$ is not a division algebra, hence is isomorphic to $M_{2}(\mathbb{Q}_p)$, i.e. $B$ splits (or is unramified) over $p$. Hence, $\mathfrak{e} = \mathbb{Z}$.
            \item Because $t$ is square-free, $\mathfrak{a}=t\mathbb{Z}$.
            \item Because $K$ is imaginary quadratic, $\infty$ ramifies, so $\nu=1$ and $N_{K/\mathbb{Q}}(\mathcal{O}_K^{\times}) = \{1\}$, so $\left[\mathbb{Z}^{\times}:N_{K/\mathbb{Q}}(\mathcal{O}_K^{\times})\right]=2$.\qedhere
        \end{itemize}
        \end{proof}
\end{theorem}

The above Theorem gives us the number $\displaystyle{\#\left(L_0\left[1, \frac{1}{2}\mathbb{Z}\right]/\Gamma(L_0)\right)}$ from Proposition \ref{genus-classes} and the fact that in this case, $H(\xi)$ is commutative (see proof of \cite[Theorem 5.10]{shimura_diophantine_2006}). We are interested in the cases when this number is $1$.\\

For the different $\mathfrak{d}_{K/\mathbb{Q}}$, we have
\begin{equation*}
\mathfrak{d}_{K/\mathbb{Q}}=
    \begin{cases}
        2\sqrt{-t}\mathcal{O}_K & \textup{ if } -t \not\equiv 1\pmod 4\\
        \sqrt{-t}\mathcal{O}_K & \textup{ if } -t \equiv 1 \pmod 4
    \end{cases}.
\end{equation*}
We now want to determine the ideal $\mathfrak{c}$. But, by the above
\begin{equation*}
N_{K/\mathbb{Q}}(\mathfrak{d}_{K/\mathbb{Q}})=
    \begin{cases}
        4t\mathbb{Z} & \textup{ if } -t \not\equiv 1\pmod 4\\
        t\mathbb{Z} & \textup{ if } -t \equiv 1 \pmod 4
    \end{cases},
\end{equation*}
and $\phi_0[\xi]=1$, $\displaystyle{\phi_0(\xi, L_0) = \frac{1}{2}\mathbb{Z}}$. So, from \eqref{conductor}, we get the equation
\begin{equation*}
\mathfrak{c}_p^{2}\cdot
    \begin{cases}
        4t\mathbb{Z}_p & \textup{ if } -t \not\equiv 1\pmod 4\\
        t\mathbb{Z}_p & \textup{ if } -t \equiv 1 \pmod 4
    \end{cases}=4t\mathbb{Z}_p.
\end{equation*}
Therefore,
\begin{equation}\label{c}
\mathfrak{c}=
    \begin{cases}
        \mathbb{Z} & \textup{ if } -t \not\equiv 1\pmod 4\\
        2\mathbb{Z} & \textup{ if } -t \equiv 1 \pmod 4
    \end{cases}.
\end{equation}
We will now consider specific cases in order to determine when the index in \eqref{index} is 1.
\begin{itemize}
    \item $t=1$. Then, $\mathfrak{a}=\mathfrak{e}=\mathfrak{a}^{*} = \mathbb{Z}$ and from \eqref{c} $\mathfrak{c} = \mathbb{Z}$. Now
    \begin{equation*}
        U' = \{x\in \mathcal{O}_K^{\times} \mid x-1\in 2\mathcal{O}_{K,p} \textup{ }\forall p\}.
    \end{equation*}
    Clearly $\pm 1 \in U'$ but $\pm i \notin U'$ because $2\mathcal{O}_K=(1+i)^2\mathcal{O}_K = (1-i)^2\mathcal{O}_K$. So, $[U:U'] = 2$. Finally, $\mu = 0$ because $\mathfrak{a}^{*}\mathfrak{e} = \mathbb{Z}$ and therefore, we get
    \begin{equation*}
        [H(\xi)_{\mathbb{A}}:H(\xi)_{\mathbb{Q}}(H(\xi)_{\mathbb{A}}\cap C)] = 1 \cdot 2^{1-0}\cdot 2^{-1} = 1.
    \end{equation*}
    \item $t=3$. Here $\mathfrak{a} = 3\mathbb{Z}$ and $\mathfrak{a}^{*}=3\mathbb{Z}$ because $r_{3}(\xi)= -\frac{4}{3} \in \frac{1}{3}\mathbb{Z}_{3}$ but $r_{3}(\xi) \notin \mathfrak{C}_3$, as $\mathfrak{C}_3 = \mathbb{Z}_3$. By \eqref{c}, we get $\mathfrak{c} = 2\mathbb{Z}$. In this case, $\mathcal{O}_K^{\times} = \{\pm 1,\pm \omega, \pm \omega^2\}$, where $\omega = \frac{1}{2}(1 + \sqrt{-3})$. Then,
    \begin{equation*}
        U' = \{x \in \mathcal{O}_K^{\times} \mid x-1 \in 2\sqrt{-3}\mathcal{O}_{K,p} \textup{ }\forall p \neq 3\}.
    \end{equation*}
    But for $p \neq 3$, $\sqrt{-3}$ is a unit in $\mathcal{O}_{K,p}$. So, the condition becomes $x-1 \in 2\mathcal{O}_{K,p} \textup{ }\forall p \neq 3$. We can then check that this is true only for $\pm 1 \in U$. Therefore, $[U:U'] = 3$. Also, $\mu=1$ in this case, because $3$ ramifies in $K$. Also, as $-3 \equiv 5 \pmod 8$, we have that $2$ remains prime in $K$, hence
    \begin{equation*}
        [H(\xi)_{\mathbb{A}}:H(\xi)_{\mathbb{Q}}(H(\xi)_{\mathbb{A}}\cap C)] = 1 \cdot 2^{1-1}\cdot 3^{-1}\cdot 2\cdot (1+1/2) = 1.
    \end{equation*}
    \item $t=2$ or $t>3$ with $-t \nequiv 1 \pmod 4$. We write $t=p_1\cdots p_k$ with $p_i$ distinct prime factors. In this case, $\mathfrak{e}=\mathfrak{c}=\mathbb{Z}$ and $\mathfrak{a}=t\mathbb{Z}$. For all $p_i$, we have that $r_{p_{i}}(\xi)$ satisfies the conditions of Theorem \ref{rank 1 main theorem} and therefore $\mathfrak{a}^{*}=t\mathbb{Z}$. We have $\mu=k$ as each $p_i$ is ramified in $K$. Now $U = \mathcal{O}_K^{\times} = \{\pm 1\}$ and
    \begin{equation*}
        U'=\{x \in \mathcal{O}_K^{\times} \mid x - 1 \in 2\sqrt{-t}\mathcal{O}_{K,p} \textup{ }\forall p \neq p_i, i=1,\cdots, k\}.
    \end{equation*}
    But for all $p \neq p_i$, $\sqrt{-t}$ is a unit in $\mathcal{O}_{K,p}$ and therefore $\pm 1 \in U'$, i.e. $[U:U']=1$. We then obtain:
    \begin{equation*}
        [H(\xi)_{\mathbb{A}}:H(\xi)_{\mathbb{Q}}(H(\xi)_{\mathbb{A}}\cap C)] = c_K \cdot 2^{1-k}.
    \end{equation*}
    Therefore, this is $1$ iff $c_K = 2^{k-1}$. Hence, the answer in this case is the number fields $K = \mathbb{Q}(\sqrt{-t})$ that satisfy $c_{K}=2^{k-1}$, where $k$ is the number of prime factors of $t$ and $-t \nequiv 1 \pmod 4$. For example, when $c_K = 2$, $t$ must have 2 prime factors, and examples would be $t=2, 6, 10,$ etc.
    \item $t>3$ with $-t \equiv 1 \pmod 4$. We write $t=p_1\cdots p_k$ with $p_i$ primes. In these cases, similarly to the case $t=3$, we have $\mathfrak{a} = t\mathbb{Z}$, $\mathfrak{c} = 2\mathbb{Z}$ and $\mathfrak{a}^{*}=t\mathbb{Z}$. Also, $\mu=k$ because each prime $p_i$ is ramified in $K$. Now, $U = \mathcal{O}_K^{\times} = \{\pm{1}\}$ and then
    \begin{equation*}
        U' = \{x \in \mathcal{O}_K^{\times} \mid x-1 \in 2\sqrt{-t}\mathcal{O}_{K,p} \textup{ }\forall p\neq p_i, i=1,\cdots,k\}.
    \end{equation*}
    But for $p \neq p_i$, $\sqrt{-t}$ is a unit in $\mathcal{O}_{K,p}$ and therefore $\pm 1 \in U'$, which means $[U:U']=1$. Now, if $-t \equiv 5 \pmod 8$, we get that $2$ is inert in $K$ and if $-t \equiv 1 \pmod 8$, $2$ splits in $K$. Therefore
    \begin{multline*}
        \,\,\,\,\,\,\,\,\,\,\,\,\,\,\ [H(\xi)_{\mathbb{A}}:H(\xi)_{\mathbb{Q}}(H(\xi)_{\mathbb{A}}\cap C)]=\\=
        \begin{cases}
            c_K \cdot 2^{1-k} \cdot 2 \cdot (1 + 1/2) = 3c_K \cdot 2^{1-k} & \textup{ if } -t \equiv 5 \pmod 8\\
            c_K \cdot 2^{1-k} \cdot 2 \cdot (1-1/2) = c_K \cdot 2^{1-k} & \textup{ if } -t \equiv 1 \pmod 8.
        \end{cases}
    \end{multline*}
    In the first case the index cannot be $1$, while in the second case, we must have $c_{K} = 2^{k-1}$. Therefore, the answer in this case is $t$ such that $-t \equiv 1 \pmod 8$ so that if $K = \mathbb{Q}(\sqrt{-t})$ we have $c_K = 2^{k-1}$, where $k$ is the number of distinct prime factors of $t$. For example, the only example for $k=2$ is $t=15$.
\end{itemize}
Hence, we arrive at the following Proposition:
\begin{proposition}\label{class_number_1}
    Let $t$ be one of the following:
    \begin{itemize}
        \item $t=1,3$.
        \item $t \nequiv 3 \pmod 4$ and if $t=p_1\cdots p_k$, $K:=\mathbb{Q}(\sqrt{-t})$, we have $c_K = 2^{k-1}$.
        \item $t \equiv 7 \pmod 8$ and if $t=p_1\cdots p_k$, $K:=\mathbb{Q}(\sqrt{-t})$, we have $c_K = 2^{k-1}$.
    \end{itemize}
    Then, with the notation as above, we have $[H(\xi)_{\mathbb{A}}:H(\xi)_{\mathbb{Q}}(H(\xi)_{\mathbb{A}}\cap C)]=1$.
\end{proposition}
Let now $t$ be one of the above. The set of primes $\mathcal{P}$ is as described in Theorem \ref{dirichlet} and we include the prime $2$. We aim to give an Euler product expression for $\mathcal{D}_{F,\mathcal{P}_{k,-q,r}, \mathcal{P}}(s)$. In particular, from Theorem \ref{dirichlet}, we need to give an Euler product expression for
\begin{equation*}
    \zeta_{\xi, \mathcal{P}}(s)= \sum_{\substack{N=1\\(N,p)=1 \forall p\in \mathcal{P}}}^{\infty} n(\xi; N)N^{-s},
\end{equation*}
where now $n(\xi; N) = \#\{s \in \mathbb{Z}/2tN\mathbb{Z} \mid s^2 \equiv -4t \pmod{4tN}\}$, as we explained in the beginning of Section \ref{explicit examples}.\\

Now, as $t$ is square-free, we obtain from $s^2 \equiv -4t \pmod{4tN}$, that $2t \mid s$, so it suffices to look for the number of solutions of the congruence
\begin{equation*}
    ts^2 \equiv -1 \pmod{N},
\end{equation*}
with $s \pmod N$. This last number of solutions is multiplicative in $N$, so we can write
\begin{equation*}
    \zeta_{\xi,\mathcal{P}}(s) = \prod_{p \notin \mathcal{P}}\left(\sum_{k=0}^{\infty}n(\xi;p^{k})p^{-ks}\right).
\end{equation*}
Now, for all $p \notin \mathcal{P}$, we have $(p,t)=1$, so from \cite[Proposition 14]{congruences} (as we assume $2 \in \mathcal{P}$), we get
\begin{equation*}
    n(\xi;p^{k}) = 1 + \left(\frac{-t}{p}\right),
\end{equation*}
for all $k \geq 1$, where $\left(\frac{\cdot}{p}\right)$ denotes the Legendre symbol. Therefore, if we define $\chi_{t}(p):=\left(\frac{-t}{p}\right)$ for $p \not \in \mathcal{P}$, we deduce that (bear in mind that $\chi_{t}^2 = 1$)
\begin{equation}\label{zeta_xi}
    \zeta_{\xi,\mathcal{P}}(s) = \zeta_{\mathcal{P}}(s)\zeta_{\mathcal{P}}(2s)^{-1}\zeta_{\mathcal{P}}(s,\chi_{t}),
\end{equation}
where $\zeta_{\mathcal{P}}(s,\chi_{t}) := \prod_{p \not \in \mathcal{P}}(1-\chi_t(p)p^{-s})^{-1}$.
Therefore, we have the following Theorem.
\begin{theorem}\label{rank 1}
    Let $S=2t$, with $t$ being chosen as in Proposition \ref{class_number_1}. Assume $\mathcal{P}$ is a finite set of primes, containing the primes described in Theorem \ref{dirichlet}, the prime $2$, and the primes so that the conditions of \cite[Proposition 5.13]{shimura_orthogonal} are satisfied for all $p \not \in \mathcal{P}$. Moreover, for all $p \not \in \mathcal{P}$, we define
    \begin{equation*}
        \chi_{t}(p) := \left(\frac{-t}{p}\right), \ \psi(p):=\left(\frac{-1}{p}\right).
    \end{equation*}
     If $F$ and $\mathcal{P}_{k,-q,r}$ are as in Theorem \ref{dirichlet} and $\xi = (1,0,1)^{t}$ (in particular $A(\xi) \neq 0$), we have
    \begin{equation*}
        \mathcal{D}_{F,\mathcal{P}_{k,-q,r},\mathcal{P}}(s) = A(\xi)L_{\mathcal{P}}\left(F;s-k+3/2\right)\zeta_{\mathcal{P}}(2s-2k+4)^{-1}\frac{\zeta_{\mathcal{P}}(s-k+2,\chi_{t})}{\zeta_{\mathcal{P}}(s-k+2,\psi)}.
    \end{equation*}
\end{theorem}
\begin{proof}
    The proof follows by Theorem \ref{dirichlet} after choosing $f$ to be the constant $\textbf{1}$ on $H(\xi)_{\mathbb{A}}$. We have computed $\zeta_{\xi,\mathcal{P}}(s)$ in \eqref{zeta_xi} and
    \begin{equation*}
        L_{\mathcal{P}}(\textbf{1}, s-k+2) = \zeta_{\mathcal{P}}(s-k+2, \psi)\zeta_{\mathcal{P}}(s-k+2),
    \end{equation*}
    as this can be computed by \cite[Proposition 5.15]{shimura_orthogonal}, because $\mathcal{P}$ is chosen so that conditions of \cite[Proposition 5.13]{shimura_orthogonal} are satisfied.
\end{proof}
\begin{remark}
    In the case $t=1$, we recover (partly) the result of Kohnen and Skoruppa in \cite{kohnen_skoruppa}. In particular, it is clear that with the above approach, some Euler factors might be missing. However, the benefit is that we also obtain results for $t > 1$. These could be interpreted as results for modular forms on a paramodular group (cf. \cite[Corollary 6]{krieg_integral_orthogonal}, \cite{paramodular_hecke}).
\end{remark}
\subsection{The rank $n \geq 2$ case}\label{even rank}
In the rank $n \geq 2$ case, a Theorem like \cite[Theorem 5.7]{shimura_diophantine_2006} is not available. For this reason, we seek examples of matrices $S$ so that the following conditions hold:
\begin{enumerate}
    \item The lattice $L = \mathbb{Z}^{n}$ is $\mathbb{Z}$-maximal.
    \item With the notation as in Section \ref{relation to the class number}, $L_0 \cap W$ is a $\mathbb{Z}$-maximal lattice in $W$ and if $D=\{\alpha \in H(\xi)_{\mathbb{A}} \mid \alpha (L_0\cap W) = L_0\cap W\}$, we have $D = H(\xi)_{\mathbb{A}} \cap C$.
    \item The number of classes in the genus of maximal lattices (this is independent of the choice of the maximal lattice, see \cite[Paragraph 9.7]{shimura2004arithmetic}) in $H(\xi)_{\mathbb{Q}} = \hbox{SO}^{\psi}(W)$ is $1$. Here, $\psi := \phi_0 \mid _{W}$. 
\end{enumerate}
We will show that for rank $n = 2,4,6,8$, there is at least one positive definite even symmetric matrix $S$ of rank $n$, so that the above conditions are satisfied. We start with the following lemma:
\begin{lemma}\label{H group}
    We have $H(\xi)_{\mathbb{Q}} = \textup{SO}^{\psi}(W)$ and $\psi$ can be represented by the matrix
    \begin{equation}\label{matrix T}
        T = \frac{1}{2}\m{-2& \\ & -S}.
    \end{equation} 
\end{lemma}
\begin{proof}
    We have $W = \{x \in V_0 \mid \phi_0(x,\xi)=0\}$. Now $\xi \in U:= \mathbb{Q}e_1+\mathbb{Q}e_{n+2}$ and $W = (W\cap U) \oplus U^{\perp}$. But on $U^{\perp}$, $\phi$ is represented by $-S$. Moreover, $W\cap U$ has dimension $1$ and if we write $x = \lambda e_1+\mu e_{n+2} \in W\cap U$, we have $\phi_0(x,\xi) = \lambda+\mu$ and so $W\cap U$ is spanned by $e_1-e_{n+2}$. By evaluating $\phi_0[e_1-e_{n+2}]=-1$, the lemma follows.
\end{proof}
We claim the following choices for the matrix $S$ satisfy the conditions $(1)-(3)$ above.
\begin{itemize}
    \item $n=2$: $S = \m{2 & -1 \\ -1 & 2}, \m{2 &-1\\-1&8}$ with determinants $3, 15$ respectively.
    \item $n=4$: $S = \m{2 & -1 & -1 & -1 \\ -1 & 2 & 1 & 0\\-1 & 1 & 2 & 0\\-1&0&0&2}, \m{2&-1&0&0\\-1&2&0&0\\0&0&2&-1\\0&0&-1&2}$ with determinants $5,9$ respectively.
    \item $n=6$: $S = \m{2&1&-1&1&-1&1\\1&2&0&1&-1&1\\-1&0&2&-1&1&0\\1&1&-1&2&-1&0\\-1&-1&1&-1&2&0\\1&1&0&0&0&2}$ with determinant $3$.
    \item $n=8$: $S=\m{2&-1&1&1&-1&-1&1&-1\\-1&2&0&-1&0&1&-1&1\\1&0&2&1&-1&0&0&0\\1&-1&1&2&-1&-1&1&-1\\-1&0&-1&-1&2&1&-1&1\\-1&1&0&-1&1&2&-1&1\\1&-1&0&1&-1&-1&2&-1\\-1&1&0&-1&1&1&-1&2}$ with determinant $1$.
\end{itemize}
\begin{remark}
    The matrices $S$ of rank $2$ correspond to the unitary groups of the imaginary quadratic fields $\mathbb{Q}(\sqrt{-3}), \mathbb{Q}(\sqrt{-15})$ respectively (see \cite[Example 1.6.6, (v)]{eisenstein_thesis}).
\end{remark}
Let us first check the conditions $(1)$ and $(3)$, right before Lemma \ref{H group}.
Condition $(1)$ follows by \cite[Proposition 1.6.12]{eisenstein_thesis} for the matrices with square-free determinant and \cite[Lemma 1.6.5, (ii)]{eisenstein_thesis} for the matrix of determinant $9$.\\

Condition $(3)$ follows by \cite[Section 8]{hanke2011enumerating}, because for the above choices of $S$, the matrix $T$ of \eqref{matrix T} corresponds to the following quadratic forms:
\begin{itemize}
    \item $n=2$: Examples 4, 26 in matrices of $3$ variables in \cite[Section 8]{hanke2011enumerating}.
    \item $n=4$: Examples $3,5$ in matrices of $5$ variables in \cite[Section 8]{hanke2011enumerating}.
    \item $n=6$: Example $3$ in matrices of $7$ variables in \cite[Section 8]{hanke2011enumerating}.
    \item $n=8$: Example $1$ in matrices of $9$ variables in \cite[Section 8]{hanke2011enumerating}.
\end{itemize}
All these examples correspond to one class in the genus of the standard lattice $\mathbb{Z}^{n+1}$, in the cases when it is maximal. This can also be seen by a simple computation in MAGMA for example. As we mentioned above, this shows that every maximal lattice has one class in its genus.
Finally, we need to check condition $(2)$. We will use \cite[Proposition 11.12]{shimura2004arithmetic}. We have the following local result.
\begin{proposition}\label{local_class_number}
    Let $S$ denote any of the matrices above. Let $D_p = \{\alpha \in H(\xi)_{p} \mid \alpha (L_{0,p}\cap W_p) = L_{0,p}\cap W_p\}$ for each prime $p$. We then have that $L_{0,p} \cap W_p$ is a $\mathbb{Z}_p$-maximal lattice in $W_p$ and also
    \begin{equation*}
        D_p = H(\xi)_{p} \cap C_p,
    \end{equation*}
    where $C_p = \{x \in G_p^{*} \mid xL_{0,p} = L_{0,p}\}$.
\end{proposition}
\begin{proof}
    Our proof is based on \cite[Proposition 11.12]{shimura2004arithmetic}. We will first establish the following claim:
    \begin{center}
    $L_{0,p}^{*} \neq L_{0,p} \iff p \mid \det (S)$.
    \end{center}
    This follows from the fact that $L_{0,p}^{*} = S_0^{-1}L_{0,p}$ and that 
    \begin{equation*}
        S_0^{-1} = \m{&&1\\&-S^{-1}&\\1&&&}, \textup{ } S^{-1} = \frac{1}{\det(S)}\textup{adj}(S).
    \end{equation*}
    We remind ourselves that $\phi_0[\xi]=1$ and $\phi_0(\xi, L_{0,p}) = \frac{1}{2}\mathbb{Z}_p$ for all primes $p$. Therefore, $\phi_0(\xi, L_{0,p})^{2} = \phi_0[\xi]\mathbb{Z}_p$ for all $p \neq 2$. Moreover, $L_{0,2}^{*}=L_{0,2}$ for every choice of $S$, as $2 \nmid \det(S)$ for any $S$. This means \cite[Proposition 11.12]{shimura2004arithmetic} is applicable in every case.\\\\
    Let now $t_p$ denote the dimension of the maximal anisotropic subspace of $(\mathbb{Q}_p^{n+2}, \phi_0)$ (cf. \cite[Paragraph 8.3]{shimura2004arithmetic}). For $p \nmid \det(S)$, we have $L_{0,p}^{*}=L_{0,p}$, $t_p \neq 1$ as $n$ is even and $4\phi_0[\xi]^{-1}\phi_0(\xi, L_{0,p})^{2}=\mathbb{Z}_p$. So, by \cite[Proposition 11.12, (iii), (2)]{shimura2004arithmetic}, $L_{0,p} \cap W_p$ is $\mathbb{Z}_p$-maximal.\\\\
    If now $p \mid \det(S)$, we claim $t_p > 1$ and therefore \cite[Proposition 11.12, (iii), (1)]{shimura2004arithmetic} will be applicable. We show this case by case. Define $K_0 := \mathbb{Q}_p(\sqrt{\delta})$, where $\delta := (-1)^{(n+2)(n+1)/2}\det(\phi_0)$. We note that from the proof of \cite[Lemma 3.3]{shimura_classification_quadratic_forms}, we have $t_p=2$ if and only if $K_0 \neq \mathbb{Q}_p$. 
    \begin{itemize}
        \item $S = \m{2&-1\\-1&2}$. Then, we claim that $t_3 = 2$. Now, $\det(\phi_0) = -3/2^{4}$ and then $K_0 =\mathbb{Q}_{3}(\sqrt{\det(\phi_0)}) = \mathbb{Q}_3(\sqrt{-3}) \neq \mathbb{Q}_3$.
        \vspace{0.1cm}
        \item $S = \m{2&-1\\-1&8}$. In this case, $\det(\phi_0) = -15/2^{4}$ and we claim $t_3 = t_5 = 2$. But again, if $p \in \{3,5\}$, we have $K_0 = \mathbb{Q}_p(\sqrt{\det{\phi_0}}) = \mathbb{Q}_p(\sqrt{-15}) \neq \mathbb{Q}_p$ by taking valuations (for example, if $\sqrt{-15} \in \mathbb{Q}_3$, then $-15 = u^2$ for some $u \in \mathbb{Q}_3$ and so $2v_{3}(u)=1$, contradiction).
        \vspace{0.1cm}
        \item $S = \m{2 & -1 & -1 & -1 \\ -1 & 2 & 1 & 0\\-1 & 1 & 2 & 0\\-1&0&0&2}$. In this case, $\det(\phi_0) = -5/2^{6}$ and so $K_0 = \mathbb{Q}_5(\sqrt{5}) \neq \mathbb{Q}_5$, so $t_5=2$.
        \vspace{0.1cm}
        \item $S = \m{2&-1&0&0\\-1&2&0&0\\0&0&2&-1\\0&0&-1&2}$. In this case, $K_0 = \mathbb{Q}_3$, and using a software (e.g. SageMath) we can compute $t_3=4$.
        \vspace{0.1cm}
        \item $S = \m{2&1&-1&1&-1&1\\1&2&0&1&-1&1\\-1&0&2&-1&1&0\\1&1&-1&2&-1&0\\-1&-1&1&-1&2&0\\1&1&0&0&0&2}$. In this case, $\det(\phi_0) = -3/2^{8}$ and so again $K_0 = \mathbb{Q}_3(\sqrt{-3}) \neq \mathbb{Q}_3$. Therefore, $t_3=2$.
        \vspace{0.1cm}
        \item In the $n=8$ case, we have that $\det(S)=1$, so there are no primes to check ($L_{0,p}^{*}=L_{0,p}$ for all primes $p$).
    \end{itemize}
    Finally, the fact that $D_p = H(\xi)_p\cap C_p$ follows from \cite[Proposition 11.12, (iv)]{shimura2004arithmetic}, as $t_p \neq 1$ always, because $n$ is even.
\end{proof}
We are now ready to give the global statement. 
\begin{proposition}
    With $S$ be any of the matrices above, we have that $L_0 \cap W$ is $\mathbb{Z}$-maximal in $W$ and if $D=\{\alpha \in H(\xi)_{\mathbb{A}} \mid \alpha(L_0\cap W) = L_0\cap W\}$, then $D = H(\xi)_{\mathbb{A}} \cap C$.
\end{proposition}
\begin{proof}
    The first claim follows by \cite[Proposition 1.6.9]{eisenstein_thesis}, as all the localizations are maximal by Proposition \ref{local_class_number}. For the second one, we have
    \begin{equation*}
        H(\xi)_{\mathbb{A}} \cap C = \{x \in H(\xi)_{\mathbb{A}} \mid xL_0 = L_0\}.
    \end{equation*}
    But the lattice $xL_0$ is the lattice which is defined by $(xL_0)_p = x_pL_{0,p}$ for all primes $p$. Now, if $x \in H(\xi)_{\mathbb{A}}$ with $xL_0 = L_0$, we have $x(L_0\cap W) = L_0\cap W$ (see \cite[page 104]{shimura2004arithmetic}), so $H(\xi)_{\mathbb{A}} \cap C \subset D$.\\\\
    On the other hand, for all primes $p$, let $D_p := D \cap H(\xi)_p$. Then, if $x \in D$, then $x_p \in D_p$, so $x_p \in H(\xi)_{p} \cap C_p$ by Proposition \ref{local_class_number}. Therefore, $x_pL_{0,p} = L_{0,p}$ for all primes $p$. This means that $x \in H(\xi)_{\mathbb{A}} \cap C$, as wanted.
\end{proof}
\subsection{Euler product expression for the Dirichlet series}
The question of this Section is to obtain an Euler product expression for the Dirichlet series of interest in each of the above cases and relate it to standard $L$-function attached to $F$. Again, let $\mathcal{P}$ be as in Theorem \ref{dirichlet}, containing also the prime $2$. In particular, $\mathcal{P}$ contains the primes factors of $q$, hence of $\det S$ as well. By Theorem \ref{dirichlet}, the first step is to determine $\zeta_{\xi, \mathcal{P}}(s)$. Hence, we need to compute the quantity
\begin{equation*}
    n(\xi;d) = \#\left\{s \in \mathbb{Z}^{n}/dS\mathbb{Z}^{n} \mid \frac{1}{2}qS^{-1}[s] \equiv -q \pmod{qd}\right\},
\end{equation*}
with $\xi = (1,\mathbf{0},1)^{t}$, as we have explained in the beginning of Section \ref{explicit examples}. The steps we follow are:
\begin{enumerate}
    \item Find unimodular integer matrices $P,Q$ such that $PSQ = \textup{diag}(a_1,\cdots, a_n)$, for some positive integers $a_i$.
    \item We then substitute $t = Ps \implies s = P^{-1}t$. Then, we have
    \begin{equation*}
        s-s' \in dS\mathbb{Z}^{n} \iff t-t' \in dPS\mathbb{Z}^{n} \iff t-t' \in dPSQ\mathbb{Z}^{n},
    \end{equation*}
    because $Q$ is unimodular. Hence, if $t=(t_1, \cdots, t_n)^{t}$, we consider each $t_i \pmod {da_i}$.
    \item We then solve the congruence
    \begin{equation*}
        \frac{1}{2}qS^{-1}[P^{-1}t] \equiv -q \pmod{qd}.
    \end{equation*}
\end{enumerate}
Let us now deal with the specific examples we have. In the following, let for $p \not \in \mathcal{P}$
\begin{equation*}
    \chi_{S}(p):=\left(\frac{(-1)^{n/2}\det{S}}{p}\right).
\end{equation*}
\begin{enumerate}
    \item $S = \m{2 & -1 \\ -1 & 2}$. We have $PSQ = \textup{diag}(1,3)$, with $P = \m{1&1\\1&2}$ and specified $Q$. We then have
    \begin{equation*}
        P^{-1}t = \m{2t_1-t_2\\-t_1+t_2}
    \end{equation*}
    and after substituting, the congruence of interest becomes ($q=3$ here)
    \begin{equation*}
        3t_1^2-3t_1t_2 + t_2^{2} \equiv -3 \pmod {3d}
    \end{equation*}
    with $t_1 \pmod{d}$ and $t_2 \pmod{3d}$. Now, by the form of the equation, we get $3 \mid t_2$, so the congruence becomes
    \begin{equation*}
        t_1^{2}-3t_1t_2 + 3t_2^2 \equiv -1 \pmod {d}
    \end{equation*}
    with $t_1, t_2 \pmod{d}$.  Now, we have $n(\xi;d) = N(T;d)$, where the last quantity is defined as the number of solutions $t=(t_1,t_2) \pmod{d\mathbb{Z}^2}$ to the congruence $T[t] \equiv -1 \pmod{d}$, with
    \begin{equation*}
        T = \m{1 & -3/2 \\ -3/2 & 3}.
    \end{equation*}
    But $N(T;d)$ is multiplicative in $d$. Let now $p \notin \mathcal{P}$, so that $2$ has a multiplicative inverse $\Mod{p}$ and $p \nmid \det{T}$ ($\mathcal{P}$ contains the prime factors of $\det S$ by assumption). Then, by \cite[Corollary 1]{diagonalisation}, we know that for each $k \geq 1$, there is a non-singular $\pmod{p}$ matrix $U_k$ such that $T[U_k] \equiv R \pmod {p^{k}}$ with some diagonal matrix $R$. By then setting $t \longmapsto U_kt$, we can still consider $t_{i} \pmod{p^{k}}$ for all $i$ (as the determinant of $U_k$ is non-zero $\pmod{p}$). We then have $N(T;p^{k}) = N(R; p^{k})$, with $R$ diagonal. We can now count the number of solutions $N(R;p^{k})$ by \cite[Proposition 4]{congruences}. In particular, if $R = \textup{diag}(a_1,a_2)$, we have
    \begin{equation*}
        N(R;p^{k}) = p^{k}\left[1 - \frac{1}{p}\left(\frac{-a_1a_2}{p}\right)\right].
    \end{equation*}
    But $\det{R}=a_1a_2$, and $\det{R} \equiv (\det{U_k})^{2}\det{T} \pmod{p}$, so
    \begin{equation*}
        \left(\frac{-a_1a_2}{p}\right) = \left(\frac{-\det{T}}{p}\right) = \left(\frac{-3u^{2}}{p}\right) = \left(\frac{-3}{p}\right) = \chi_{S}(p),
    \end{equation*}
    where $2u\equiv 1 \pmod{p}$. Therefore, we obtain
    \begin{equation*}
        n(\xi;p^{k}) = N(R;p^{k}) = p^{k}\left[1-\frac{\chi_{S}(p)}{p}\right].
    \end{equation*}
    \item $S = \m{2&-1 \\ -1 & 8}$. In this case, we have $PSQ = \textup{diag}(1, 15)$ with $P = \m{1&1\\1&2}$ and specified $Q$. By following the above steps, the congruence becomes
    \begin{equation*}
        15t_1^2-15t_1t_2+4t_2^2 \equiv -15 \pmod{15d}
    \end{equation*}
    with $t_1 \pmod{d}$ and $t_2 \pmod{15d}$. But now $15 \mid t_2$ and so after $t_2 \longmapsto 15t_2$, we get
    \begin{equation*}
        t_1^2-15t_1t_2+60t_2^2 \equiv -1 \pmod{d}
    \end{equation*}
    with $t_1,t_2 \pmod{d}$. With exact same reasoning as before, we get
    \begin{equation*}
        n(\xi;p^{k}) = p^{k}\left[1-\frac{\chi_{S}(p)}{p}\right],
    \end{equation*}
    for all $p \notin \mathcal{P}$ with $n(\xi;d)$ multiplicative in $d$.
    \vspace{0.1cm}
    \item $S = \m{2 & -1 & -1 & -1 \\ -1 & 2 & 1 & 0\\-1 & 1 & 2 & 0\\-1&0&0&2}$. In this case, we have $PSQ = \textup{diag}(1,1,1,5)$ with
    \vspace{-0.1cm}
    \begin{equation*}
        P = \m{3&1&1&3\\1&1&0&1\\1&0&1&1\\3&1&1&4}.
    \end{equation*}
    By then substituting $s = P^{-1}t$ the congruence becomes
    \begin{equation*}
        5t_1^{2}+5t_2^{2}+5t_3^{2}+2t_4^{2}-5t_1t_2-5t_1t_3-5t_1t_4+5t_2t_3 \equiv -5 \pmod{5d},
    \end{equation*}
    with $t_1,t_2,t_3 \pmod{d}$ and $t_{4} \pmod{5d}$. But, again, $5\mid t_4$, so after setting $t_4 \longmapsto 5t_4$, we have
    \begin{equation*}
        t_1^{2}+t_2^{2}+t_3^{2}+10t_4^{2}-t_1t_2-t_1t_3-5t_1t_4+t_2t_3 \equiv -1 \pmod{d},
    \end{equation*}
    with $t_i \pmod{d}$ for all $i$. As before, this then gives
    \begin{equation*}
        n(\xi;p^{k}) = p^{3k}\left[1-\frac{\chi_{S}(p)}{p^{2}}\right],
    \end{equation*}
    for all $p \notin \mathcal{P}$ with $n(\xi;d)$ multiplicative in $d$.
    \vspace{0.1cm}
    \item $S = \m{2&-1&0&0\\-1&2&0&0\\0&0&2&-1\\0&0&-1&2}$. In this case, using the same method, we obtain
    \begin{equation*}
        3t_1^2+12t_2^2+t_3^2+t_4^2-9t_1t_2+3t_1t_3-6t_2t_3 + 3t_2t_4 \equiv -3 \pmod{3d},
    \end{equation*}
    with $t_1,t_2 \pmod {d}$ and $t_3,t_4 \pmod{3d}$. This then implies $3 \mid t_3^2+t_4^2$ and so $3 \mid t_3,t_4$. Therefore, as previously, we obtain
    \begin{equation*}
        n(\xi;p^{k}) = p^{3k}\left[1-\frac{\chi_{S}(p)}{p^{2}}\right],
    \end{equation*}
    for all $p \notin \mathcal{P}$ with $n(\xi;d)$ multiplicative in $d$.
    \item In the last two cases, we omit the calculations, but by applying the exact same reasoning as above, we will again get
    \begin{equation*}
        n(\xi;p^{k}) = p^{k(n-1)}\left[1-\frac{\chi_{S}(p)}{p^{n/2}}\right],
    \end{equation*}
    for all $p \not \in \mathcal{P}$ and with $n(\xi;d)$ multiplicative in $d$. Here $n=6, 8$.
\end{enumerate}
In all of the above cases, we therefore obtain (here $\mathcal{P}$ can be possibly enlarged but still finite)
\begin{equation}\label{zeta_xi even}
    \zeta_{\xi, \mathcal{P}}(s) = \prod_{p \notin \mathcal{P}} \zeta_{\mathcal{P}}(s-(n-1))\zeta_{\mathcal{P}}(s-n/2+1, \chi_{S})^{-1},
\end{equation}
where $\zeta_{\mathcal{P}}(s, \chi_S) := \prod_{p \not \in \mathcal{P}}(1-\chi_S(p)p^{-s})^{-1}$.
We therefore arrive at the following Theorem.
\begin{theorem}\label{even rank theorem}
    Let $S$ denote any of the matrices of Section \ref{even rank} with even rank $n$ and $\mathcal{P}$ a finite set of primes, containing the primes described in the proof of Theorem \ref{dirichlet}, the prime $2$, and the primes so that the conditions of \cite[Proposition 5.13]{shimura_orthogonal} are satisfied for all $p \not \in \mathcal{P}$. For all $p \not \in \mathcal{P}$, we define 
    \begin{equation*}
    \chi_{S}(p):=\left(\frac{(-1)^{n/2}\det{S}}{p}\right).
    \end{equation*}
    If $F$ and $\mathcal{P}_{k,-q,r}$ are as in Theorem \ref{dirichlet} and $\xi = (1,\mathbf{0},1)^{t}$ (in particular $A(\xi) \neq 0$), we have
    \begin{multline*}
        \mathcal{D}_{F,\mathcal{P}_{k, -q, r},\mathcal{P}}(s) = A(\xi)\zeta_{\mathcal{P}}(2s-2k+n+2)^{-1}\zeta_{\mathcal{P}}\left(s-k+(n+4)/2,\chi_{S}\right)^{-1}\times \\\times L_{\mathcal{P}}\left(F;s-k+(n+2)/2\right)\prod_{i=1}^{n-1}\zeta_{\mathcal{P}}(s-k+(n+2)-i)^{-1}.
    \end{multline*}
\end{theorem}
\begin{proof}
    The proof follows from Theorem $\ref{dirichlet}$ after we choose $f$ to be the constant $\textbf{1}$ on $H(\xi)_{\mathbb{A}}$. From \cite[Proposition 5.15]{shimura_orthogonal}, we have that (since $\mathcal{P}$ is chosen so that conditions of \cite[Proposition 5.13]{shimura_orthogonal} are satisfied):
        \begin{equation*}
            L_{\mathcal{P}}\left(\mathbf{1}, s-k+(n+3)/2\right) = \prod_{i=1}^{n}\zeta_{\mathcal{P}}(s-k+(n+2)-i),
        \end{equation*}
    as in general $L_{\textup{Sug}}\left(\textbf{1}; s-(n-1)/2\right) = L_{\textup{Shi}}(s)$. Here, the notation means the $L$-function we encounter in \cite{sugano} and \cite{shimura_orthogonal} respectively, as these are normalised differently. Finally, the computation of $\zeta_{\xi, \mathcal{P}}(s)$ has been performed in \eqref{zeta_xi even}.
\end{proof}
\renewcommand{\abstractname}{Acknowledgements}
\begin{abstract}
The author would like to thank Prof. Thanasis Bouganis for suggesting this problem as well as for his continuous guidance and support. The author would also like to thank the anonymous referees for their useful comments and suggestions.
\end{abstract}
\vspace{0.5cm}
\textbf{Funding}: This work was supported by the HIMR/UKRI "Additional Funding Programme for Mathematical Sciences", grant number EP/V521917/1, as well as by a scholarship from the Onassis Foundation.\\


\vspace{-0.2cm}
\bibliographystyle{plain}
\bibliography{reference}
\end{document}